\documentclass{amsart}
\usepackage{all2018}
\usepackage{pxfonts}
\usepackage{algorithm2e}
\newcommand{\ZZo}{\ZZ_{\geq 0}}

\newcommand{\Gm}{{\mathbb G}_{\mathrm m}}
\newcommand{\Ga}{{\mathbb G}_{\mathrm a}}
\newcommand{\EE}{\mathbb{E}}
\renewcommand{\phi}{\varphi}
\newcommand{\uphi}{\underline{\phi}}
\newcommand{\upsi}{\underline{\psi}}
\newcommand{\sat}{\mathrm{sat}}
\begin{document}

\title{Matroids over one-dimensional groups}

\author{Guus P.~Bollen}
\address{Department of Mathematics and Computer Science,
Eindhoven University of Technology, P.O.~Box 513, 5600 MB,
Eindhoven, The Netherlands}
\email{g.p.bollen@tue.nl}

\author{Dustin Cartwright}
\address{Department of Mathematics,
University of Tennessee, 227 Ayres Hall, Knoxville, TN
37996-1320, USA}
\email{cartwright@utk.edu}

\author{Jan Draisma}
\address{Mathematisches Institut, Universit\"at Bern,
Sidlerstrasse 5,
3012 Bern, Switzerland; and Department of Mathematics and Computer
Science, Eindhoven University of Technology, P.O.~Box 513,
5600 MB Eindhoven, The Netherlands}
\email{jan.draisma@math.unibe.ch}

\begin{abstract}
We develop the theory of matroids over one-dimensional algebraic groups,
with special emphasis on positive characteristic. In particular, we
compute the Lindstr\"om valuations and Frobenius flocks of such matroids.
Building on work by Evans and Hrushovski, we show that the class of
algebraic matroids, paired with their Lindstr\"om valuations, is not
closed under duality of valuated matroids.
\end{abstract}

\maketitle

\section{Introduction}

Given  an algebraically closed field $K$ and a natural number $n$, any
irreducible subvariety $X \subseteq K^n$ gives rise to a matroid $M(X)$
on the ground set $[n]:=\{1,\ldots,n\}$ by declaring $I \subseteq
[n]$ to be independent if the coordinate projection $X \to K^I$ is
dominant, i.e., has dense image. The variety $X$ is an {\em algebraic
representation} 
of $M(X)$ and $M(X)$ is called an {\em algebraic matroid}.
For an equivalent, but more algebraic interpretation, one considers the
field $K(X)$ of rational functions on $X$, and calls $I$ independent
if the coordinates $x_i,\ i\in I$ map to elements of $K(X)$ that are
algebraically independent over $K$. For an introduction to algebraic
matroids we refer to \cite{Rosen19}.

The best-understood algebraic matroids are {\em linear matroids},
which are those coming from linear subvarieties $X \subseteq K^n$.
However, the class of algebraic matroids is larger than just
linear matroids. For example, in the first paper to study algebraic
representability of matroids, Ingleton gave an algebraic representation
of the non-Fano matroid, over any field, using a variety $X \subseteq
K^7$ parametrized by monomials~\cite[Ex.~15]{Ingleton71}, which we would
now call a toric variety, the Zariski closure of a connected algebraic
subgroup of $(K^*)^n$---a subtorus. More generally, such parametrizations
can be used to show that any linear matroid over $\QQ$ is algebraic over
any field. This construction is reviewed in Example~\ref{ex:Toric}.

In fact, linear spaces and subtori are examples of connected subgroups of
$G^n$, where $G$ is an arbitrary connected, one-dimensional algebraic
group over $K$.  The possibilities for $G$ are the additive group
$\Ga = (K, +)$, the multiplicative group $\Gm = (K^*, \cdot)$,
and any elliptic curve over $K$. The purpose of this paper is to
develop a unified theory for the algebraic matroids arising from such
subgroups. In this theory, the linear parametrization of a linear space
and the monomial parametrization of a subtorus are replaced by any
homomorphism of algebraic groups $\Psi:G^d \rightarrow G^n$ with image
$X$, a closed and connected subgroup of $G^n$. The homomorphism $\Psi$
is described by an $n \times d$ matrix with elements in the ring $\EE$
of endomorphisms of the group $G$. For instance, with $G=(K^*,\cdot)$ we
have $\EE \cong \ZZ$ via the isomorphism $\ZZ \ni a \mapsto (t \mapsto
t^a)$, and the rows of the matrix record the exponent vectors in the
monomial parametrization $\Psi$. If $K$ has characteristic zero, then
the endomorphism $\EE$ is commutative for any connected, one-dimensional
algebraic group $G$ over $K$. But if $K$ has positive characteristic,
then $\EE$ can be non-commutative; see \S\ref{ssec:Endo} for details.
These non-commutative rings give examples of non-linear matroids which are
algebraic over all fields of positive characteristic~\cite{Lindstrom86a}.

All of our results are formulated
uniformly over the three different types of algebraic groups, and only
in the proofs do we sometimes distinguish between them.

For all one-dimensional groups $G$,
$\EE$ is both a left and right Ore domain, so it is contained in a
division ring $Q$ (generated by $\EE$), and the algebraic matroid
represented by $X$ has a linear representation over $Q$. Here, we write
$M(X)$ for the matroid whose bases are all sets $I \subseteq [n]$ such
that the projection of $X$ to $G^I$ is all of $G^I$. This matroid is
equivalent to the algebraic matroid defined in the first paragraph
because, if we choose a non-constant rational function $h \colon G
\dto K$ and let $Y \subset K^n$ be the coordinatewise image
$h^n(X)$, then $M(X) = M(Y)$.

\begin{figure}
\begin{center}
\includegraphics{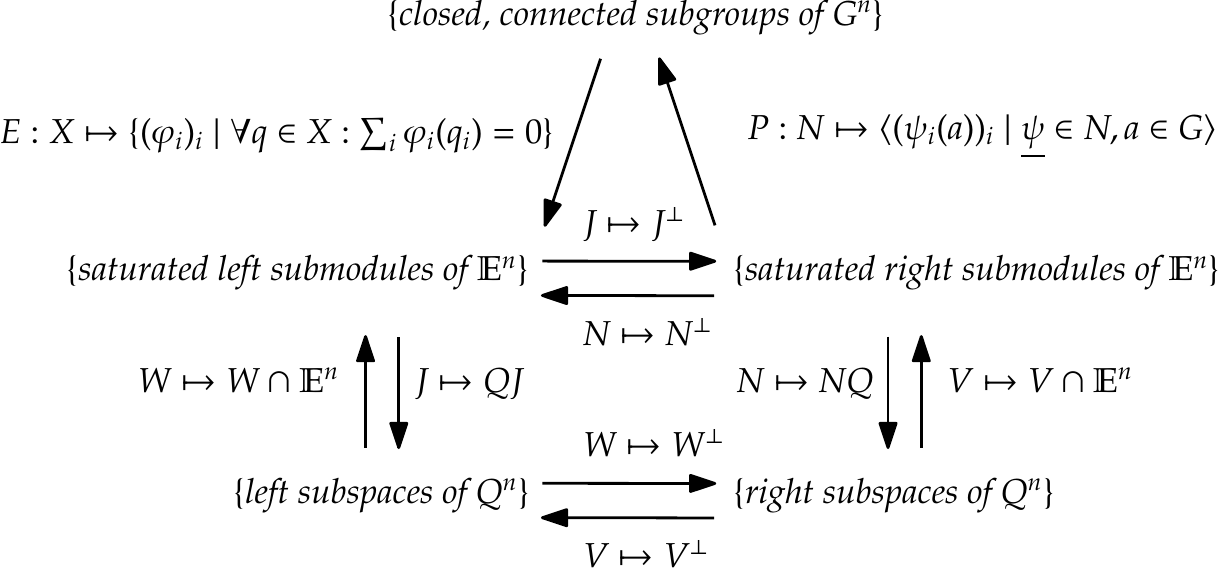}
\end{center}
\caption{The diagram of Theorem~\ref{thm:bijections}.}
\label{fig:bijections}
\end{figure}

\begin{thm} \label{thm:bijections}
Let $G$ be a connected, one-dimensional algebraic group
over an algebraically closed field $K$. Then the maps in
Figure~\ref{fig:bijections} are bijections and the diagram commutes.
Furthermore, for a closed, connected subgroup $X \subseteq G^n$ the
set $M(X)$ is a matroid and coincides with the linear matroid
on $[n]$ defined by the right vector space $P^{-1}(X)Q \subset Q^n$.
\end{thm}


The map $P$ in Figure \ref{fig:bijections} sends $N$ to the subgroup of
$G^n$ generated by the elements of the form $(\psi_i(a))_i$ with $a \in
G$ and $\upsi$ running through $N$. We think of the right modules $N$
and $V$ as giving {\em parametrizations} of $X$ and of the left modules
$J$ and~$W$  as giving {\em equations} for $X$, whence the notation $E$
and $P$. The ring $\EE$ is left and right Noetherian, so a right module
$N \subseteq \EE^n$ has a finite generating set. Use these vectors as
the columns of an $n \times d$-matrix $\Psi$. Then $\Psi$ gives a
natural group homomorphism $G^d
\to G^n$; $P(N)$ is the image of this homomorphism. A different choice
of generators yields a different matrix $\Psi$ with the same column
space and a different homomorphism $G^{d'} \to G^n$ with the same image.

We note that our use of the column space of an $n \times m$
matrix~$\Psi$, in order to define a matroid on the ground set $[n]$, differs
from the conventional use of row spaces in matroid theory, which define
matroids on the set of columns. However, in our construction of $\Psi$,
it is the rows which are labeled by $[n]$, because $\Psi$ is a matrix
defining a group homomorphism to $G^n$. Since $\EE$ is possibly
non-commutative, the column space of $\Psi$ is not equivalent to the row
space of its transpose, and so we use the column space to define our
matroids consistently.


Because of Theorem~\ref{thm:bijections}, we call the matroids isomorphic to $M(X)$
for some closed, connected $X \subseteq G^n$ {\em $\EE$-linear}.
An $\EE$-linear matroid $M$ admits an algebraic
representation in the sense of the first paragraph of this paper:
choosing a non-constant rational function $h:G \dto K$ defined near $0
\in G$ we obtain a rational map $h^n: G^n \dto K^n$, and the variety
$Y:=\overline{h^n(X)} \subseteq K^n$ has $M(Y)=M$.

\begin{thm} \label{thm:duality}
The class of $\EE$-linear matroids is closed under contraction, deletion,
and duality.
\end{thm}

\begin{ex} \label{ex:KF}
Let $G=\Ga$ be the additive group over $K = \overline{\FF_2}$.
Then $\EE$ is isomorphic to the skew polynomial ring $K[F]$
in which multiplication is governed by the rule $Fa=a^2 F$.
The (right) column space $N$ of the matrix
\[\Psi = \left(
    \begin{array}{cc}
      1 & 0 \\
      0 & 1 \\
      1 & 1 \\
      1 & F \\
    \end{array}
  \right),
\]
is saturated in $K[F]^4$, and $P(N)=\{(a,b,a+b,a+b^2) \mid (a,b) \in
\Ga^2 \}.$ The matroid $M(P(N))$ is the uniform matroid $U_{2,4}$.
\hfill $\clubsuit$
\end{ex}

\begin{ex} \label{ex:Toric}
Let $G=\Gm$ be the multiplicative group over $K$.  Then $\EE \cong
\ZZ$ via the map $\ZZ \to \EE, a \mapsto (t \mapsto t^a)$. Any matrix
$\Psi \in \ZZ^{n \times d}$ of full rank $d$ gives rise to an algebraic
group homomorphism $\Psi: \Gm^d \to \Gm^n, t=(t_1,\ldots,t_d) \mapsto
(t_1^{\psi_{i1}} \cdots t_d^{\psi_{id}})_{i=1}^n$ whose image $X$
is a $d$-dimensional subtorus of the $n$-dimensional torus $\Gm^n$.
The matroid $M(X)$ is equal to the linear matroid over $\QQ$ in which
$I \subseteq [n]$ is an independent set if and only if the corresponding
rows of $\Psi$ are linearly independent.

This classical argument shows
that matroids linear over $\QQ$ are algebraic over any field. The group
homomorphism $\Psi: \Gm^d \to \Gm^n$ is a closed embedding if and only
if the $\ZZ$-column space of $\Psi$ is a saturated submodule of $\ZZ^n$.
Our framework is a common generalization of linear matroids
and these toric matroids.
\hfill $\clubsuit$
\end{ex}

\subsection*{Positive characteristic}
In characteristic zero, by Ingleton's theorem \cite{Ingleton71},
the class of matroids that admit an $\EE$-linear representation
is the same as the class of matroids that admit a $K$-linear
representation. Therefore, we next specialize to characteristic $p>0$,
and we study $\EE$-linear matroids from the perspective of the theory
developed in \cite{Bollen17,Cartwright18}.  There, for any irreducible
variety $Y \subseteq K^n$, a canonical matroid valuation on $M(Y)$ is
constructed, called the {\em Lindstr\"om valuation}. Furthermore, in
\cite{Bollen17}, a so-called {\em Frobenius flock} is associated to the
pair consisting of $Y$ and a sufficiently general point of
$Y$. Here, in \S~\ref{ssec:Flock},
we will define the Frobenius flock of a closed, connected subgroup $X
\subseteq G^n$, use it to define the Lindstr\"om valuation on $M(X)$,
and relate these definitions to the constructions of \cite{Bollen17} via
a rational function $G \dto K$ as above. The goal is, then, to express
these invariants of $X$ in the data of Theorem~\ref{thm:bijections}.

To this end, we proceed as follows. In \S~\ref{ssec:Valuation} we
construct a canonical valuation $v\colon\EE \to \ZZo \cup \{\infty\}$
such that $v(\alpha)>0$ if and only if $\alpha \colon G \rightarrow G$ is an inseparable
morphism. A routine check shows that $v$ extends to $Q$.
Then, in \S~\ref{ssec:VSpaceFlocks}, we describe a general construction
of a $K$-linear flock from a right vector subspace $V \subset Q^n$,
and of a compatible valuation on the linear matroid on $[n]$ determined
by $V$. These are related to the Frobenius flock and the Lindstr\"om
valuation of $X$ and to \cite{Bollen17} as follows.

\begin{thm} \label{thm:Flock}
Let $X$ be a closed, connected subgroup of $G^n$ and let
$N=P^{-1}(X)$ be the saturated right submodule of $\EE^n$ representing
$X$. Then the Frobenius flock of $X$ equals the linear flock of
$NQ$, and the Lindstr\"om valuation of $X$ equals the matroid valuation
corresponding to $NQ$. Furthermore, let $h:G \dto K$ be a
rational function defined near $0$ and set
$Y:=\overline{h^n(X)}$. Under mild conditions on $h$, the
point $h^n(0) \in Y$ satisfies the genericity condition (*) from
\cite{Bollen17}, and $d_0 h^n$ restricts to a linear
bijection between the Frobenius flock of $X$ and the
Frobenius flock of $(Y,0)$.
\end{thm}

The mild conditions in this theorem are specified in Equation (**)
in \S~\ref{ssec:ProofFlock}.

By Theorem~\ref{thm:duality}, the algebraic matroids arising
from groups are closed under duality. However, the method used to prove
this duality does not dualize the Lindstr\"om valuation in the sense
of \cite[Proposition 1.4]{Dress92}. On the contrary, using results from
\cite{Evans91}, we prove:

\begin{thm}\label{thm:NonDual}
The class of algebraic matroids equipped with their Lindstr\"om
valuations is not closed under duality of valuated matroids. In
particular, if $K$ has positive characteristic, then there exists a
closed, connected subgroup $X \subset \Ga^n$, with algebraic matroid $M$
and Lindstr\"om valuation $w$, such that the dual valuated matroid
$(M^*, w^*)$ is not the algebraic matroid and Lindstr\"om valuation of
any variety $Y \subset K^{n}$.
\end{thm}

To be clear, the dual matroid $M^*$ from Theorem~\ref{thm:NonDual} {\em
is} algebraic---indeed, this follows from Theorem~\ref{thm:duality}---but
the Lindstr\"om valuation of any realization of $M^*$ is distinct from
$w^*$.

\subsection*{Relation to existing literature}

As noted, the use of endomorphisms of $\Gm$ appears as an example
in \cite{Ingleton71}.  The systematic use of endomorphisms of $\Ga$
(dub\-bed {\em $p$-polynomials}) in matroid theory first appeared
in \cite{Lindstrom88}, though particular cases were studied in
\cite{Lindstrom86a} and \cite{Lindstrom86b}, which used $p$-polynomials to
show that the non-Pappus matroid is algebraic over any field of positive
characteristic. The first uniform treatment of endomorphism rings of
all possible one-dimensional algebraic groups is in~\cite{Evans91},
which used model theory to show that each algebraic representation
of certain matroids would have to be algebraically equivalent to an
$\EE$-linear representation for some one-dimensional algebraic group $G$
with endomorphism ring $\EE$.

There is no logical dependence between our work and the recent literature
on matroids over tracts and hyperfields \cite{Baker19}, although algebraic
realizations of matroids in positive characteristic do yield matroids
over natural {\em non-commutative} hyperfields \cite{Pendavingh18}.

Modules over rings with a valuation feature both here and in
\cite{Fink19}, but there is no immediate connection between our results
and theirs. Indeed, there, the input data consists of a finitely presented
module over $R$ for each subset of a ground set, subject to certain axioms
that enhance properties of the rank function of an ordinary matroid;
while here, a single submodule of $\EE^n$ represents a matroid over an
algebraic group with endomorphism ring $\EE$.

The rows of our matrix $\Psi$ representing a parametrization $G^d \to X
\subseteq G^n$ give homomorphisms $G^d \to G$, and their kernels define
a collection of codimension-one subgroups of $G^d$.  There is currently
much research activity on such (central) {\em Abelian arrangements},
especially in the case where the ground field is $\CC$ and the
topology of the complement is studied via combinatorial techniques
\cite{DeConcini05,Adderio13,DAli18,Delucchi19}. While the questions here
are somewhat orthogonal to our focus, the progression from hyperplane
arrangements to toric arrangements to elliptic arrangements mirrors the
observation that toric varieties provide non-linear algebraic matroids,
followed by the generalization to elliptic curves. However, our interest
is primarily in characteristic-$p$ phenomena, and especially the
non-commutative endomorphism rings that arise there.

\subsection*{Organization}

The remainder is organized as follows. In Section~\ref{sec:Arbitrary}
we recall basic facts about one-dimensional algebraic groups $G$,
their endomorphisms rings $\EE$, and submodules of $\EE^n$; and
we prove Theorems~\ref{thm:bijections} and~\ref{thm:duality}. In
Section~\ref{sec:CharP}, we zoom in on characteristic $p$ and establish
Theorem~\ref{thm:Flock}. Finally, in Section~\ref{sec:Examples} we give
several examples of $\EE$-linear matroids with $G$ equal to $\Gm$, $\Ga$,
or an elliptic curve and prove Theorem~\ref{thm:NonDual}.

\subsection*{Acknowledgments}
JD was partially supported by the NWO Vici grant entitled {\em
Stabilisation in Algebra and Geometry}. GPB was supported by the Stichting
Computer Algebra Nederland. DC was partially supported by NSA Young
Investigator grant H98230-16-1-0019. We thank Aurel Page for his
MathOverflow answer at
\texttt{https://mathoverflow.net/questions/309320}, which helped with
the last case of Lemma~\ref{lm:ValOne}. Work on this paper started at
the Institut Mittag-Leffler during their program ``Tropical Geometry,
Amoebas, and Polyhedra''; we thank the institute and its staff for
fantastic working conditions. 

%
%
%
%

\section{Arbitrary characteristic} \label{sec:Arbitrary}

\subsection{One-dimensional algebraic groups}
Let $G$ be a one-dimensional algebraic group over the algebraically
closed field $K$. Thus, $G$ is either the multiplicative group
$\Gm=(K^*,\cdot)$, the additive group $\Ga=(K,+)$, or an elliptic
curve (for a classification of the affine one-dimensional groups see
\cite[Theorem III.10.9]{Borel91}; Hurwitz' automorphism theorem rules
out that the genus of $G$ is strictly greater than $1$). In particular,
$G$ is Abelian. We write the operation in $G$ additively in the uniform
treatment of these cases.

\subsection{The endomorphism ring} \label{ssec:Endo}
The set $\EE:=\End(G)$ of endomorphisms of $G$ as an algebraic group is a
ring with operations $\cdot=\circ$ (composition) as multiplication and
$(\phi+\psi)(g):=\phi(g)+\psi(g)$ as addition; for $\phi+\psi$ to be an
endomorphism, the fact that $G$
is Abelian is crucial.

If $G=\Gm$, then $\EE \cong \ZZ$ via the map $\ZZ \to \EE,\ a \mapsto
(t \mapsto t^a)$. If $G=\Ga$ and $\cha K=0$, then $\EE \cong K$ via
the map $K \to \EE, c \mapsto (d \mapsto cd)$. For $G=\Ga$ in positive
characteristic, $\EE$ is the skew polynomial ring $K[F]$ whose elements
are polynomials in $F$ and multiplication is governed by the rule
$Fa=a^p F$ for all $a \in K$, because $a \in K$ acts by scaling $\Ga$
and $F$ acts as the Frobenius operator.  If $G$ is an elliptic curve,
then $\EE$ is either the ring of integers, an order in an imaginary
quadratic number field, or an order in a quaternion algebra (only in
positive characteristic) \cite[Theorem V.3.1]{Silverman09}. For any
one-dimensional group $G$, its endomorphism ring $\EE$ embeds into a
division ring $Q$ (generated by $\EE$): $Q$ is either a number field,
or a quaternion algebra, or $K$ itself, or the division ring $K(F)$
of the ring of $p$-polynomials constructed in \cite{Lindstrom88}.

\begin{ex} \label{ex:curve}
Consider the elliptic curve $G$ over $\overline{\FF}_2$ in the
projective plane whose equation in the affine $(x,y)$-chart is $y^2 +
y =x^3+x+1$. The Abelian group structure on $G$ is uniquely determined by
requiring that the point $O:=(0:1:0)$ is the neutral element and requiring
that the intersection points with $G$ of any line in the projective
plane add up to $O$. The ring $\EE$ is isomorphic to the ring of {\em
Hurwitz quaternions}: those quaternions $a+bi+cj+dk$ with either $a,b,c,d \in \ZZ$
or else $a,b,c,d \in 1/2 + \ZZ$; see \cite[Chapter 3, \S 6 and Appendix
IV]{Husemoeller04}, where explicit endomorphisms are described. Having
a non-commutative endomorphism ring, the curve is supersingular.
\end{ex}

\subsection{Submodules of $\EE^n$}
Since $\EE$ is, in general, a non-commutative ring, we distinguish between
left submodules $J$ and right submodules $N$ of $\EE^n$. We define $QJ$
(respectively, $NQ$) to be the left (respectively, right) $Q$-subspace
of $Q^n$ generated by $J$. Linear algebra over a division ring $Q$ works
very similarly to linear algebra over fields---all modules are free and there
are well-defined notions of basis and dimension---and this is why we use
``$Q$-vector space'' rather than ``$Q$-module''.  However, one needs
to distinguish between left and right vector spaces, and for instance
the dual of a left $Q$-vector space $V$ is naturally a right vector
space.

The dimension of $QJ$ (respectively, $NQ$) is called the {\em rank}
$\rk J$ (respectively, $\rk N$) of $J$ (respectively, $N$). We call
$J^\sat:=QJ \cap \EE^n$ and $N^\sat:=NQ\cap \EE^n$ the {\em saturations}
of $J$ and $N$, and we call $J$ and $N$ {\em saturated} if they are
equal to their saturations. These are straightforward generalizations
of the familiar notion of saturated lattice in $\ZZ^n$. In
particular, $J$ is saturated if and only if the
following holds: if $\alpha \uphi \in J$ for some $\uphi \in \EE^n$
and $\alpha \in \EE \setminus \{0\}$, then $\uphi \in J$; and similarly
for $N$. Furthermore, $J \mapsto QJ$ and $N \mapsto NQ$ are bijections
between saturated submodules of $\EE^n$ and $Q$-subspaces in $Q^n$. In
particular, if $J \subseteq J'$ are both saturated and $\rk J=\rk J'$,
then $J=J'$; and similarly for $N$.

\subsection{Orthogonal complements}\label{ssec:Orthogonal}
We have the natural pairing $\EE^n \times \EE^n \to \EE$:
\[ \langle \uphi, \upsi \rangle:=\sum_i \phi_i \psi_i. \]
This pairing is left-$\EE$-linear in the first
argument and right-$\EE$-linear in the second argument.
The orthogonal complement $J^\perp$ of a left $\EE$-submodule
$J \subseteq \EE^n$ is:
\[ J^\perp=\left\{\upsi \mid \forall \uphi \in J:
\langle \uphi, \upsi \rangle=0\right\}; \]
it is a saturated right $\EE$-module whose rank equals $n-\rk J$, and
similarly for the orthogonal complement of a right submodule
$N$. The operation $\perp$ also extends to left or right
subspaces of $Q^n$.

\subsection{Connected subgroups}

\begin{lm} \label{lm:Proj}
Let $X$ be a closed, connected subgroup of $G^n$ of dimension $d$. Then
there exist surjective algebraic group homomorphisms $\alpha\colon X \to
G^d$ and $\beta\colon G^d \to X$ and a natural number $e$ such that $\alpha
\circ \beta$ and $\beta \circ \alpha$ are the multiplication with $e$
maps on $G^d$ and $X$, respectively. Moreover, if $G=\Ga$ or $G=\Gm$,
then $e$ can be taken equal to $1$.
\end{lm}

\begin{proof}
For $G=\Gm$, this follows from~\cite[Proposition III.8.5]{Borel91}. For $G=\Ga$ in characteristic zero, a closed, connected
subgroup of $\Ga^n$ is just a linear subspace of $K^n$, and the result
is basic linear algebra.  For $G=\Ga$ in positive characteristic, the
result follows from~\cite[Lemma B.1.10]{Conrad10}.

Now assume that $G$ is an elliptic curve. Then the lemma follows from
the fact that isogeny of Abelian varieties is an equivalence relation. We
recall some details. Take any maximal subset $I \subseteq [n]$ such that
the projection $\alpha: X \to G^I$ is dominant. Since algebraic group
homomorphisms have a closed image, $\alpha$ is surjective. Maximality
implies that $|I|=d$, so that $H:=\ker \alpha$ is finite.  Let $e>0$ be
the exponent of the finite group $H$. Then the multiplication map $\gamma:
X \to X, x\mapsto ex$ has $\ker \gamma \supseteq H$. Therefore $\gamma$
factors as $\beta \circ \alpha$ for some algebraic group homomorphism
$\beta:G^I \cong X/H \to X$. So $\beta \circ \alpha$ is multiplication
by $e$. Conversely, for $a \in G^I$ there exists an $x \in X$ such that
$\alpha(x)=a$, and we find
\[
\alpha(\beta(a))=\alpha(\beta(\alpha(x)))=\alpha(ex)=e\alpha(x)=ea,
\]
so also $\alpha \circ \beta$ is multiplication by $e$.  Since
multiplication by $e$ is surjective---here we use that $X$ is
connected---both $\alpha$ and $\beta$ are surjective.
\end{proof}

\begin{re}
The proof in the latter paragraph also works in the cases $G=\Gm$, except
that it does not yield $e=1$. It does not work for $G=\Ga$, since the
exponent $e$ of $H$ in the proof might be $p$, so that multiplication
with $e$ is not surjective.
\end{re}

\subsection{Proof of Theorems~\ref{thm:bijections} and~\ref{thm:duality}}

The following lemmas establish Theorem~\ref{thm:bijections}.

\begin{lm}
The maps $P$ and $E$ in the diagram are well-defined.
\end{lm}

\begin{proof}
We begin with $E$. Let $X$ be a closed, connected subgroup of $G^n$ and define
\[ J:=E(X)=\{\uphi \in \EE^n \mid \forall q \in X: \sum_i \phi_i(q_i)=0 \}. \]
A straightforward computation shows that $J$ is a left $\EE$-submodule of
$\EE^n$. To show that $J$ is saturated, suppose that $\alpha \uphi \in
J$ with $\alpha \in \EE \setminus \{0\}$ and $\uphi \in \EE^n$. Then
for each $q \in X$, $\uphi(q):=\sum_i \phi_i(q_i)$ is in the kernel
of the endomorphism $\alpha$. This kernel is a closed subset of the
one-dimensional variety $G$, and since $\alpha \neq 0$, we find that
$\ker \alpha$ is a finite set of points. Since $X$ is irreducible and
since $\uphi$, regarded as a map $G^n \to G$, is a morphism, $\uphi(X)
\subseteq \ker \alpha$ is a single point. This point is $0$ since
$\phi(0)=0$. Hence $\uphi \in J$, and $J$ is saturated as desired.

Next we show that $P$ is well-defined. Let $N$ be any right
submodule of $\EE^n$ (not necessarily saturated) and let
$\upsi^{(1)},\ldots,\upsi^{(m)}$ be a generating set of $N$.  Write
$\Psi=(\upsi^{(1)},\ldots,\upsi^{(m)})$. We think of $\Psi$ as an $n
\times m$-matrix over $\EE$. It gives rise to the group homomorphism
\begin{equation} \label{eq:Par}
\Psi: G^m \to G^n,\quad (a_1,\ldots,a_m) \mapsto
\sum_{j=1}^m \upsi^{(j)}(a_j),
\end{equation}
whose image $X$ is closed, connected subgroup of $G^n$. A straightforward
computation shows that $X=P(N)$.

The two maps at the bottom in the diagram are clearly well-defined
if $J$ is any left $\EE$-submodule of $\EE^n$, then $J^\perp$ is a
saturated right-submodule of $\EE^n$, and vice versa.
\end{proof}

\begin{lm}
The maps $J \mapsto J^\perp$ and $N \mapsto N^\perp$ (when
restricted to saturated modules) are inverse to each other.
\end{lm}

\begin{proof}
The module $(J^\perp)^\perp$ is a saturated
left $\EE$-submodule of $\EE^n$ which on the one hand contains $J$
and on the other hand has the same rank as $J$. Hence they
are equal. The same argument applies to $N$.
\end{proof}

\begin{lm} \label{lm:P}
For any right $\EE$-submodule $N$ of $\EE^n$,
$P(N)=P(N^\sat)$, and $\dim P(N)=\rk N$.
\end{lm}

\begin{proof}
From $N^\sat \supseteq N$ we immediately find $P(N^\sat) \supseteq
P(N)$. Conversely, for $\upsi \in N^\sat$ there exists an $\alpha \in
\EE$ such that $\upsi \alpha \in N$, and since the map $\alpha\colon G \to G$
is surjective, $\im (\upsi \alpha\colon G \to G^n)=\im (\upsi\colon G \to G^n)$. This
proves the first statement.

This shows that $P(N)=P(N')$ where $N' \subseteq N$ is any submodule
of $N$ of the same rank generated by vectors that are (right-)linearly
independent over $Q$. The parametrization~\eqref{eq:Par} shows that $\dim
X \leq \rk N'=\rk N$. For the converse write $d:=\rk N$. Then exists a
subset $I \subseteq [n]$ with $|I|=d$ such that the projection of $N$
in $\EE^I$ has rank $d$, and one finds vectors $v_i \in N, i\in I$
such that $v_i$ has nonzero entry $\alpha_i$ in position $i$ and zero
entries in positions $j \in I \setminus \{i\}$. Then the image of $X$
in $G^I$ contains the elements $(\delta_{ij} \alpha_i(G))_{j \in I}$
for every $i \in I$, and these generate $G^I$. So $\dim X \geq d$.
\end{proof}

\begin{lm}\label{lm:commutes}
The diagram in Theorem~\ref{thm:bijections} commutes.
\end{lm}

\begin{proof}
We concentrate on the upper triangle; the lower square
was discussed in \S~\ref{ssec:Orthogonal}. Let $X \subseteq G^n$ be a closed and connected subgroup. By
Lemma~\ref{lm:Proj}, $X$ is the image of some homomorphism $\beta:G^d
\to G^n$ where $d=\dim X$. Then $\beta=(\beta_1,\ldots,\beta_n)$
where $\beta_i:G^d \to G$ is a homomorphism. Write $\beta_{ij}$
for the composition of the embedding $G \to G^d, a \mapsto
(0,\ldots,0,a,0,\ldots,0)$ (with $a$ in the $j$-th position) and
$\beta_i$. Then $\beta=(\beta_{ij})_{ij} \in \EE^{n \times d}$; let
$N'$ be the right submodule of $\EE^n$ generated by the columns of
$\beta$. Then $X=P(N')$ by \eqref{eq:Par}, and if we write
$N=(N')^\sat$, then $X=P(N')=P(N)$ by Lemma~\ref{lm:P}.

Also by Lemma~\ref{lm:P}, $\rk N=d$. Set $J:=N^\perp$, a left module
of rank $n-d$. The orthogonality to $N$ implies that $J$ is contained
in $E(X)$. Let $m \geq n-d$ be the rank of $E(X)$. Then there is a
subset $I \subseteq [n]$ of size $m$ such that the projection of $E(X)$
in $\EE^I$ has rank $m$. Then $E(X)$ contains, for every $i \in I$,
an element of the form $\alpha_i e_i + \sum_{j \in [n] \setminus I}
\alpha_{ij} e_j$, where $e_i$ is the $i$-th standard basis vector of
$\EE^n$. This implies that the projection $X \to G^{[n] \setminus I}$
is finite-to-one. In particular, $d=\dim X \leq n-m\leq d$, so equality
must hold everywhere and $m=n-d$ and $E(X)=J$. Thus
\[ X=P(N)=P((N^\perp)^\perp)=P(J^\perp)=P(E(X)^\perp). \]
To see that the triangle also commutes when we start at some saturated
left $\EE$-module $J \subseteq \EE^n$, set $X:=P(J^\perp)$ and note that
$\dim X=n-\rk J$ (by Lemma~\ref{lm:P}) and $J \subseteq E(X)$. If
$E(X)$ were strictly
larger than $J$, then, since they are both saturated, $\rk E(X)>\rk(J)$, but then
$\rk(E(X)^\perp)<\rk(J^\perp)$ and
\[ n-\rk J=\dim X=\dim(P(E(X)^\perp))=\rk(E(X)^\perp)<n-\rk
J, \]
a contradiction (in the second equality we use that the
triangle commutes
when starting at $X$). The same reasoning applies when we start at some
saturated right $\EE$-module $N \subseteq \EE^n$.
\end{proof}

\begin{proof}[Proof of Theorem~\ref{thm:bijections}.]
By Lemma~\ref{lm:commutes}, the diagram of Figure~\ref{fig:bijections}
commutes. Now, let $X \subseteq G^n$ be a closed, connected subgroup, and set
$N:=P^{-1}(X) \subseteq \EE^n$. Let $I \subseteq [n]$, let $X_I$ be the
image of $X$ in $G^I$ ($X_I$ is a closed, connected subgroup), and $N_I$ be the
saturation of the image of $N$ in $\EE^I$.  Now we have $X_I=P(N_I)$ and
hence $\dim X_I= \rk(N_I)=\dim_Q N_I Q$ by Lemma~\ref{lm:P}.  This proves
that $M(X)$ is (a matroid and) equal to the linear matroid on $[n]$
defined by $N Q$.
\end{proof}

\begin{proof}[Proof of Theorem~\ref{thm:duality}.]
Let $A=M(X)$ for some closed, connected subgroup $X \subseteq
G^n$. Let $N:=E(X)^\perp \subseteq \EE^n$. Then the right
$Q$-vector space $V:=NQ$ of $Q^n$ determines the same matroid $A$ by
Theorem~\ref{thm:bijections}. Now the results on deletion and contraction
follow from linear algebra over $Q$.

For duality, we argue that $Q$ has an anti-automorphism $\tau$. When
$Q$ is commutative, we take $\tau=1_Q$. When $G=\Ga$ in characteristic
$p>0$, there is a unique anti-isomorphism $\tau:\EE=K[F] \to K[F^{-1}]
\subseteq Q$ that sends $F$ to $F^{-1}$ and that is the identity on
$K$. This extends uniquely to an anti-automorphism of $Q$ ($\EE$
itself does not have a natural anti-automorphism in this case!). When $G$ is a
supersingular elliptic curve in characteristic $p>0$, we use the dual
isogeny \cite[III.9, item (ii)]{Silverman09}.

Now $V^\perp$ is a left vector space over $Q$ determining the
matroid dual to $A$, and $\tau(V^\perp) \subseteq Q^n$ a right vector space
also determining the matroid dual to $A$. Let $N':=\tau(V^\perp) \cap \EE^n$.
This is a saturated right module, and the group $X := P(N')$ determines
the matroid dual to $E$.
\end{proof}

\begin{ex}\label{ex:DualU24}
Take $M$ and $\Psi$ as in Example~\ref{ex:KF}.  The orthogonal complement
$M^\perp$ is the left submodule of $K[F]^4$ given as the row space of
the matrix
\[\Psi^\perp = \left(
    \begin{array}{cccc}
      1 & 1 & 1 & 0 \\
      1 & F & 0 & 1 \\
    \end{array}
  \right).
\]
The recipe in the proof of Theorem~\ref{thm:duality}
involves taking the right $Q$-vector space spanned by the
columns of
\[
\begin{pmatrix}
1 & 1 \\
1 & F^{-1} \\
1 & 0 \\
0 & 1
\end{pmatrix}
\]
and intersecting with $K[F]^4$. This yields the right $K[F]$-module
generated by the columns of
\[
\begin{pmatrix}
1 & F \\
1 & 1 \\
1 & 0 \\
0 & F
\end{pmatrix}.
\]
The corresponding subgroup is $\{(a+b^2,a+b,a,b^2) \mid
(a,b) \in K^2\}$, a representation of the dual of $U_{2,4}$,
which of course is also $U_{2,4}$. \hfill $\clubsuit$
\end{ex}

\section{Linear flocks from vector spaces over division rings}
 \label{sec:Flocks}

In this section we introduce some of the key structures featuring in
the statement and proof of Theorem~\ref{thm:Flock}.

\subsection{Linear flocks}
\label{ssec:VSpaceFlocks} {\em Frobenius flocks} are collections of
vector spaces which are an enrichment of a valuated matroid, introduced
in~\cite{Bollen17}. Here, we describe a slight generalization called
linear flocks, where the Frobenius automorphism is replaced
by an arbitrary automorphism~\cite[Ch. 4]{Bollen18}.  Throughout what follows, we write $e_i$
for the $i$-th standard basis vector in $\ZZ^n$.

\begin{de}
Let $L$ be a field and let $\phi \colon L \rightarrow L$ an automorphism. Then a
{\em $(L,\phi)$-linear flock} is a collection of vector spaces
$(V_\alpha)_{\alpha \in \ZZ} \in L^n$, of the same dimensions, such that
\begin{enumerate}
\item For any $\alpha \in \ZZ^n$ and $1 \leq i \leq n$, $V_\alpha / i =
V_{\alpha + e_i} \setminus i$ (here $/i$ stands for intersecting with
the $i$-th coordinate hyperplane and $\setminus i$ stands for setting
the $i$-th coordinate zero); and
\item For any $\alpha \in \ZZ^n$, $V_{\alpha + (1,\ldots,1)} =
\phi(V_{\alpha})$, where $\phi$ refers to the coordinate-wise application
of the automorphism to $L^n$.
\end{enumerate}
\end{de} 

A Frobenius flock is an $(L,F^{-1})$-linear flock where $F$ is the
Frobenius automorphism of a perfect field $L$ of positive characteristic. As
shown in \cite[Thm. 34]{Bollen17}, an algebraic variety over $L$ defines a
Frobenius flock.

\subsection{Linear flocks from right vector spaces}
We now construct a linear flock from a right vector space, as follows.
Let $Q$ be a division ring with a surjective, discrete valuation
$v\colon Q \to \ZZ \cup \{\infty\}$. Valuations are well-known in
commutative algebra, but they can also be generalized to non-commutative
rings~\cite{Schilling45}. We will only use the case of value
group~$\ZZ$.

\begin{de}
A \emph{valuation} on a ring $R$ is a
function $v \colon R \rightarrow \ZZ \cup \{\infty\}$ such that:
\begin{itemize}
\item For all $a, b \in R$, $v(ab) = v(a) + v(b)$.
\item For all $a, b \in R$, $v(a + b) \geq \min\{v(a), v(b) \}$.
\end{itemize}
\end{de}

Let $R \subset Q$ be the
valuation ring, which is the set of elements with non-negative
valuation. The residue division ring $L$ is the quotient of $R$ by the
maximal ideal consisting of all elements of positive valuation. For
simplicity, we assume that $L$ is commutative, and hence a field. We
define the automorphism $\phi \colon L \rightarrow L$ by sending the
residue class of $x \in R$ to the residue class of $\pi x \pi^{-1}$,
where $\pi$ is any \emph{uniformizer}, i.e., element of valuation 1.
Thus, if $R$ is commutative, then $\phi$ is the identity.

\begin{lm}
If $Q$ is a division ring with a valuation, whose residue division ring
$L$ is commutative, then the automorphism $\phi \colon L \rightarrow L$
is well-defined.
\end{lm}

\begin{proof}
Let $\pi'$ be another element of $R$ such that $v(\pi') = 1$ and let
$x'$ be another element of $R$ with the same residue class as $x$. Thus,
$\pi' = u\pi$, where $u$ is a unit in~$R$, and $x' = x + t$, where
$v(t) > 0$. Then, using $\overline y$ to denote the residue class in $L$
of $y \in R$,
\begin{equation*}
\overline{\pi' x' (\pi')^{-1}} = \overline{u \pi (x + t) \pi^{-1}
u^{-1}} =
\overline{u} \cdot \overline{\pi x \pi^{-1}} \cdot \overline{u^{-1}} + \overline{u
\pi t \pi^{-1} u^{-1}} = \overline{\pi x \pi^{-1}},
\end{equation*}
because $L$ is commutative and $v(u \pi t
\pi^{-1} u^{-1}) = v(t) > 0$.
\end{proof}

We recall the following lemma, and include a proof for completeness.

\begin{lm} \label{lm:Reduction}
Let $V \subseteq Q^n$ be a right vector space, and let $v_1,\ldots,v_r
\in V \cap R^n$ map to a basis of the $L$-vector space $\overline{V \cap
R^n} \subseteq L^n$. Then $v_1,\ldots,v_r$ are a $Q$-basis of $V$. In
particular, $\dim_Q V = \dim_L \overline{V \cap R^n}$.
\end{lm}

\begin{proof}
Choose $v_{r+1}, \ldots, v_n \in R^n$ such that their reductions, $\overline{v_1}, \ldots,
\overline{v_n}$ form a basis of $L^n$. By Nakayama's Lemma, $v_1,
\ldots, v_n$ generate $R^n$, and therefore they generate $Q^n$ as well,
so $v_1,\ldots,v_n$ is a basis both for $Q^n$ as a right vector space,
and for $R^n$ as a right $R$-module.

Now let $w \in V$, so that $w = v_1 a_1 + \cdots +
v_n a_n$ for some $a_1, \ldots, a_n \in Q$. Set
\[ u := v_{r+1} a_{r+1} + \cdots + v_n a_n
= w-v_1 a_1 - \cdots - v_r a_r \in V.\]
Assume that $u$ is
non-zero. Then, by scaling by an appropriate power of $\pi$, we can
assume that the minimum valuation of the coordinates of $u$ is $0$,
so $u \in R^n$ and $\overline{u}$ is non-zero. Thus, $u$ can be written
as an $R$-linear combination of $v_1, \ldots, v_n$, but we already know
that $u$ is uniquely written as $v_{r+1} a_{r+1} + \cdots + v_n a_n$,
which means that $a_{r+1}, \ldots, a_n$ must be in $R$. Therefore,
$\overline{u}$ is a nonzero linear combination
$\overline{v_{r+1}}, \ldots, \overline{v_n}$, which implies that
$\overline{u}$ is not in $\overline{V \cap R^n}$, a contradiction.
Therefore, $u$ is $0$, so $w$ is in the span of $v_1, \ldots, v_r$.
This shows that $\dim_Q V = r = \dim_L \overline{V \cap R^n}$,
as desired.
\end{proof}

Now fix a right vector space $V \subset Q^n$. For
any $\alpha \in \ZZ^n$, we define 
\begin{equation*}
V_\alpha := \overline{(\pi^{-\alpha} V) \cap R^n} \subset L^n,
\end{equation*}
where $\pi^{-\alpha}$ denotes the diagonal matrix
with entries $\pi^{-\alpha_1}, \ldots, \pi^{-\alpha_n}$.

\begin{lm} \label{lm:Flock}
The vector spaces $(V_\alpha)_{\alpha \in \ZZ^n}$ defined above form
a $(L,\phi^{-1})$-linear flock.
\end{lm}

\begin{proof}
First, for each $\alpha$ we have
\[ \dim_L V_\alpha = \dim_Q \pi^{-\alpha} V = \dim_Q V =:r, \]
where the first equality follows from Lemma~\ref{lm:Reduction}.
Second, for $1 \leq i \leq n$ and $\alpha \in \ZZ^n$ we have
\[ (\pi^{-\alpha} V) \cap R^n \supseteq \pi^{e_i} ((\pi^{-\alpha-e_i}
V) \cap R^n). \]
Applying $\pi^{e_i}$ to a vector in $R^n$ and then reducing is the
same thing as first reducing and then setting the $i$-th coordinate
to zero, so we find that $V_\alpha / i \supseteq V_{\alpha+e_i}
\setminus i$. If the latter vector space has dimension $r$,
then, since the former vector space has dimension at most $r$,
the two spaces are equal. Otherwise, $V_{\alpha+e_i} \setminus i$
has dimension $r-1$. Then $V_{\alpha+e_i}$ contains the $i$-th
standard basis vector of $L^n$ and therefore $\pi^{-\alpha-e_i}V$
contains an element of the form $(a_1,\ldots,1,\ldots,a_n)$ with
$v(a_j)>0$ for $j \neq i$. Then $\pi^{-\alpha} V$ contains the vector
$(\pi^{-1} a_1, \ldots,1,\ldots,\pi^{-1} a_n)$ whose reduction has a nonzero
$i$-th coordinate, hence $V_\alpha/i$ has dimension $r-1$, as well,
so that again we have $V_{\alpha}/i=V_{\alpha+e_i} \setminus i$.

Third, every coordinate of every element of $\pi((\pi^{-\alpha} V) \cap
R^n)$ has positive valuation, and so $\pi((\pi^{-\alpha V}) \cap R^n)
\pi^{-1}$ is contained in $R^n$, and thus equal to $(\pi^{-\alpha +
(1, \cdots,1)} V) \cap R^n$. Thus, $V_{\alpha-(1, \ldots, 1)}$ is equal
to $\phi(V_{\alpha})$.
\end{proof}

By considering the matroid associated to each $V_\alpha$, a
$(L,\phi^{-1})$-linear flock defines a matroid flock, a notion
cryptomorphic to that of a valuated matroid (up to adding scalar
multiples of the all-one vector) by \cite[Theorem 7]{Bollen17}.

We will now show that for the flock from Lemma~\ref{lm:Flock}, this
valuated matroid is the natural non-commutative generalization of the
well-known construction of a matroid valuation from a vector space of a
field with a non-Archimedean valuation. For this we recall that the {\em
Dieudonn\'e determinant} is the unique group homomorphism $\det: \GL_r(Q)
\to Q^*/[Q^*,Q^*]$ that sends a diagonal matrix $\diag(c,1,\ldots,1)$
to the class of $c$ and matrices that differ from the identity matrix only in
one off-diagonal entry to $1$. We define the Dieudonn\'e determinant
of a non-invertible square matrix to be the symbol $O$, and we make
$Q^*/[Q^*,Q^*] \cup \{O\}$ into a commutative monoid by $a \cdot O:=O$.
Since commutators have valuation $0$, the valuation $v$ induces a group
homomorphism $v:Q^*/[Q^*,Q^*] \to \ZZ$, and we set $v(O)=\infty$. The
Smith normal form algorithm shows that $r \times r$-matrix $A \in
R^{r \times r}$ has $v(\det(A)) \geq 0$ with equality if and only if
$\overline{A} \in L^{r \times r}$ is invertible.

\begin{prop}\label{prop:Determinant}
Let $V \subset Q^n$ be a right vector space, and let $\mu$ denote the
matroid valuation corresponding to the matroid flock associated to the
linear flock $(V_\alpha)_\alpha$ (defined using any uniformizer
$\pi \in R$). Then the rank in $\mu$ of a subset $I \subseteq [n]$ is
the dimension of the projection of $V$ into $Q^I$; and the
valuated circuits are:
\[
\{(v(c_1), \ldots, v(c_n)) \mid c_i \in Q, [c_1 \cdots c_n] V = 0,
\mbox{and the support of $c$ is minimal}\}.
\]
Moreover, if $V$ is given as the right column span of an $n \times r$ matrix
$A$, then the valuation $\mu(B)$ of $B \in \binom{[n]}{r}$ is
equal to the valuation $v(\det A[B])$ of the Dieudonn\'e determinant of
the submatrix of $A$ consisting of the rows labeled by $B$.  
\end{prop}

\begin{proof}
By \cite[Theorem 7]{Bollen17}, the valuation $\mu$ can be characterized
as follows (again, modulo scalar multiples of the all-one vector):
for all $\alpha \in \ZZ^n$ the expression $\mu(B)-e_B^T \alpha$,
where $e_B:=\sum_{i \in B} e_i$ is the characteristic vector of $B$,
is minimized by $B_0 \in \binom{[n]}{r}$ if and only if $B_0$ is a basis
for the matroid $M(V_\alpha)$ on $[n]$ defined by $V_\alpha$. It suffices
to prove that the numbers $v(\det A[B])$ have this property.

Suppose that $B_0$ is a basis for $M(V_\alpha)$. Choose
$v_1,\ldots,v_r \in (\pi^{-\alpha} V) \cap R^n$ such that
$\overline{v_1},\ldots,\overline{v_r}$ are a basis of $V_\alpha$.
Let $g \in \GL_r(Q)$ be the unique matrix such that $A':=\pi^{-\alpha}
A g$ has columns $v_1,\ldots,v_r$. Then for any $B \in \binom{[n]}{r}$
 we have
\begin{align*}
v(\det(A[B]))-e_B^T \alpha
&= v(\det(\pi^{\alpha}[B]) \det(A'[B]) \det(g^{-1})) - e_B^T \alpha\\
&= v(\det(A'[B])) + v(\det(g^{-1})) + (e_B^T \alpha - e_B^T \alpha)\\
&\geq v(\det(A'[B_0])) + v(\det(g^{-1})) \\
&=v(\det(A[B_0]))-e_{B_0}^T \alpha,
\end{align*}
where the inequality follows because $A'$ has entries in $R$,
and $v(\det(A'[B_0]))=0$ since $\overline{A'[B_0]}$ is invertible.
Moreover, equality holds if and only if also $B$ is a basis in the matroid
$M(V_\alpha)$. This proves the last statement in the lemma, and also that
$B$ is a basis in the underlying matroid of $\mu$ if and only if $A[B]$ is
invertible; this, in turn, implies the statement about the rank function.

Finally, recall that the valuated circuits corresponding to the basis
valuation $\mu$ are the vectors $\gamma \in (\ZZ \cup \{\infty\})^n$
that are supported precisely on some circuit $C \subseteq [n]$ of the
matroid underlying $\mu$ and have the property that for each $i \in C$,
and for each $j \in C - i:=C \setminus \{i\}$ we have
\[ \mu(C-i)+\gamma_j = \mu(C-j) + \gamma_i. \]
Without loss of generality, we may assume that $C-i = [r]$ and that $i=r+1$.
After performing column operations on $A$ over $Q$, we may assume that
$A[C-i]$ is the identity matrix, so that $\mu(B)=0$. 

Then the unique
linear relation among the rows of $A[[r+1]]$ is the vector
$a=(a_{r+1,1},\ldots,a_{r+1,r},-1,0,\ldots,0)$. We have, for
each $j=1,\ldots,r$, 
\[ \mu(C-i)+v(a_{r+1,j})=v(a_{r+1,j})=v(\det
A[C-j])=\mu(C-j )+v(-1), \]
which means that $v(a)$ satisfies the condition on $\gamma$ above. This
proves the statement about valuated circuits.
\end{proof}

\section{Lindstr\"om valuations and Frobenius flocks} \label{sec:CharP}

The goal of this section is to prove Theorem~\ref{thm:Flock}.  Let $p$ be
a fixed prime number, $K$ an algebraically closed field of characteristic
$p$, and $G$ a one-dimensional algebraic group over $K$. As before, $\EE$
denotes the endomorphism ring of $G$. We will construct a valuation on
$\EE$, introduce Frobenius flocks and Lindstr\"om valuations of closed,
connected subgroups of $G^n$, and relate these to the constructions of
Section~\ref{sec:Flocks}.

\subsection{The valuation on $\EE$} \label{ssec:Valuation}
In order to define the valuation $v \colon \EE \rightarrow \ZZ \cup \{\infty\}$, we
let $\alpha \in \EE$ be a nonzero endomorphism of $G$ and then obtain an
injective homomorphism $\alpha^*\colon K(G) \to K(G)$
of function fields. If we let $L$ be the set of elements of $K(G)$ which
are purely inseparable over $\alpha^*K(G)$, then $L/\alpha^*K(G)$ is a
purely inseparable extension of fields, and $K(G) / L$ is a separable
extension. The degree $[L : \alpha^* K(G)]$ is a power of $p$, and is
called the inseparable degree of the extension $K(G)/\alpha^*K(G)$, and
denoted $[K(G) : \alpha^* K(G)]_i$ \cite[\S V.6]{Lang02}. We define our
valuation by $v(\alpha) = \log_p [K(G) \colon \alpha^* K(G)]_i$ for
$\alpha$ non-zero, and
$v(0) = \infty$.

The main technical point to proving that $v$ is a valuation is the
following:

\begin{lm}\label{lm:InsepSubfield}
If $K'$ is a subfield of $K(G)$ such that $K(G)/K'$ is purely inseparable,
then $K' = F^n K(G)$ for some integer $n$, where
$F$ is the Frobenius endomorphism of $K(G)$.
\end{lm}

\begin{proof}
Let $x$ and $y$ be two elements in $K(G) \setminus K$. Since $K(G)/K$
has transcendence rank 1, $x$ and $y$ are algebraically dependent,
meaning that they satisfy a polynomial relation $f(x,y)$, where $f \in
K[X,Y]$. If every exponent of $f$ is divisible by $p$, then $f = g^p$,
using the fact that $K$ is algebraically closed. We can assume that $f$
is irreducible, which means that at least one exponent is not divisible
by $p$. Therefore, either $x$ is separable over $K(y)$ or $y$ is
separable over $K(x)$. Thus, the purely inseparable subfields of $K(G)$
are totally ordered by inclusion.

The fields $K(G) \supset FK(G) \supset F^2K(G) \supset \ldots$ also form
a chain of purely inseparable subfields of $K(G)$. Each containment has
index $p$, so there can't be any intermediate fields. Since $K(G)/K'$ is
purely inseparable, then $K' \neq K$, and so $K'(G)/K$ has finite index.
Therefore, $K'$ can't be contained in all the $F^nK(G)$, but since the
purely inseparable subfields are totally ordered, then it must be equal
to $F^nK(G)$ for some~$n$.
\end{proof}

\begin{prop}
The function $v \colon \EE \rightarrow \ZZ_{\geq 0} \cup \{\infty\}$
defines a valuation on $\EE$ that extends uniquely to $Q$.
\end{prop}

\begin{proof}
In order to show that $v$ is a valuation on $\EE$, we first note that
$v(\alpha \beta) = v(\alpha) + v(\beta)$ because of the multiplicativity
of inseparable degree \cite[Cor.~V.6.4]{Lang02}.

Second, we want to show that if $\alpha$ and $\beta$ are elements of
$\EE$, then $v(\alpha + \beta) \geq \min\{v(\alpha), v(\beta)\}$. By
Lemma \ref{lm:InsepSubfield}, $\alpha^* K(G)$ and $\beta^* K(G)$ are contained
in $F^{v(\alpha)} K(G)$ and $F^{v(\beta)} K(G)$, respectively, and
thus both are contained in $F^{\min\{v(\alpha), v(\beta)\}} K(G)$.
Since $(\alpha+\beta)^* K(G)$ is contained in the compositum of
$\alpha^* K(G)$ and $\beta^* K(G)$, then it is contained in $F^{\min\{
v(\alpha, v(\beta)\}} K(G)$, and thus $v(\alpha + \beta) \geq
\min\{v(\alpha), v(\beta)\}$.

Next we want to show that $v$ extends to a unique valuation on $Q$,
which is an exercise in working with rings of fractions and the Ore
condition. We refer to \cite[Chapter 1]{Cohn95} for an introduction to
Ore domains. Every element of $Q$ can be written as a fraction $\phi^{-1} \psi$ for
some $\phi, \psi \in \EE$, and we define $v(\phi^{-1} \psi) = -v(\phi) +
v(\psi)$. If we write $\phi^{-1} \psi$ instead as $(\tau\phi)^{-1}
\tau\psi$ for some $\tau \in \EE$, then
\begin{equation*}
v((\tau\phi)^{-1} \tau \psi) = -v(\tau\phi) + v(\tau\psi)
= -v(\tau) - v(\phi) + v(\tau) + v(\psi)
= -v(\phi) + v(\psi) = v(\phi^{-1} \psi),
\end{equation*}
which shows that this valuation is well-defined. Since the valuation is
defined using elements of $\EE$ and their reciprocals, this
is the unique valuation which extends the one on $\EE$.

Now we want to show that $v$ defines a valuation on $Q$.
Again, we write $\phi^{-1} \psi$ and $\sigma^{-1} \tau$ for elements of $Q$, where
$\phi$, $\psi$, $\sigma$, and $\tau$ are in $\EE$. Then we take
their product by finding $\psi'$ and $\sigma'$ in $\EE$ such that
$\sigma' \psi = \psi' \sigma$ (the existence of such elements is the
left Ore condition), and then $(\phi^{-1} \psi) (\sigma^{-1}
\tau) = (\sigma' \phi)^{-1} \psi' \tau$. The valuation of this product  is:
\begin{align*}
v((\sigma' \phi)^{-1} \psi' \tau) &= -v(\sigma' \phi) + v(\psi'\tau)
= -v(\sigma') -v(\phi) + v(\psi') + v(\tau) \\
&= -v(\phi) - v(\sigma) + v(\psi) + v(\tau) = v(\phi^{-1}\psi) +
v(\sigma^{-1}\tau),
\end{align*}
where the first equality on the second line is by our assumption that
$\sigma' \psi = \psi' \sigma$.

Second, to compute the sum of $\phi^{-1} \psi$ and $\sigma^{-1} \tau$,
we find $\phi'$ and $\sigma'$ in $\EE$ such that $\sigma' \phi = \phi'
\sigma$ and then:
\begin{equation*}
\phi^{-1} \psi + \sigma^{-1} \tau = (\sigma' \phi)^{-1} (\sigma' \psi +
\phi' \tau).
\end{equation*}
If we take the valuation of this sum, we get:
\begin{align*}
v((\sigma' \phi)^{-1} (\sigma' \psi + \phi' \tau)) &= -v(\sigma') -
v(\phi) + v(\sigma' \psi + \phi' \tau) \\
&\geq -v(\sigma') - v(\phi) + \min\{v(\sigma') + v(\psi), v(\phi') +
v(\tau)\} \\
&= \min\{-v(\phi) + v(\psi), -v(\sigma') -v(\phi) + v(\phi') + v(\tau)\}
\\
&= \min\{-v(\phi) + v(\psi), -v(\sigma) + v(\tau)\} \\
&= \min\{v(\phi^{-1} \psi), v(\sigma^{-1} \tau)\},
\end{align*}
which completes the proof the $v$ is a valuation on $\EE$.
\end{proof}

The set of all elements with positive valuation is a two-sided ideal
$I$ in $\EE$. If $G$ is defined by equations with coefficients in
$\FF_p$, such as $\Ga$ or $\Gm$, then the Frobenius homomorphism itself
is an endomorphism of $G$. Moreover, this element $F$ generates $I$ as
either a left, right or two-sided ideal. When $G$ is an elliptic curve
not defined over $\FF_p$, then $I$ need not be principal.

\subsection{Description of the valuation on $\EE$}

Recall that an elliptic curve in positive characteristic is called
\emph{supersingular} if its ring of endomorphisms is non-commutative,
and thus an order in a quaternion algebra \cite[\S V.3]{Silverman09}.

\begin{prop}\label{prop:Class}
For each connected one-dimensional algebraic group $G$, the valuation
on $\EE$ is as follows:
\begin{enumerate}
\item If $G \cong \Ga$, so that $\EE$ is the ring of $p$-polynomials
$K[F]$, then $v$ is the $F$-adic valuation.
\item If $G \cong \Gm$, so that $\EE \cong \ZZ$, with $F$ corresponding
to $p$, then $v$ is the $p$-adic valuation.
\item If $G \cong E$, an elliptic curve with $j$-invariant
not in $\overline{\FF_p}$, then $\EE \cong \ZZ$ and $v$ is the $p$-adic
valuation.
\item If $G \cong E$, a non-supersingular elliptic curve with
$j$-invariant in $\overline{\FF_p}$, then $\EE$ is an order in a quadratic
number field $\QQ(\sqrt{-D})$. Let $\mathcal O \supset \EE$ denote the
ring of integers in $\QQ(\sqrt{-D})$. Then there exists a maximal ideal
$m \subset \mathcal O$ such that $m \overline m = (p)$, where $\overline
m$ denotes complex conjugation, and $v$ is the
restriction of the $m$-adic valuation.
\item \label{item:quaternion} If $G \cong E$, a supersingular elliptic curve with
$j$-invariant in $\overline{\FF_p}$ then $\EE$ is an
order in a quaternion algebra, and $v(\alpha)$ is the $p$-adic valuation
of $\alpha \overline \alpha$.
\end{enumerate}
\end{prop}

\begin{proof}
The first two cases follow from the fact that in each case Frobenius is
an endomorphism of $G$, corresponding to $p$ and $F$, respectively.

In case (3), we know that multiplication by $p$, a morphism of degree
$p^2$ by \cite[Thm III.6.2(d)]{Silverman09}, is inseparable but not
purely inseparable, and hence $v(p)=1$. Thus, $v$ must be the $p$-adic
valuation.

In case (4), we know that any valuation on $\QQ(\sqrt{-D})$ corresponds
to a maximal ideal $m$ of $\mathcal O$, and the multiplication by $p$
endomorphism is inseparable, but not purely inseparable so $v(p) = 1$.
Thus, it remains to show that the rational prime $p$ splits in $\mathcal
O$. This is a standard fact in the theory of elliptic curves, but we
include a proof for convenience.

We define $G^{(p^i)}$ to be the elliptic curve 
obtained by applying the $i$th power of Frobenius
to the defining equations of the elliptic curve $G$. Since $G$ is
defined over $\overline \FF_p$, $G^{(p^k)} \cong G$ for some positive
integer $k$, thus composition with the $k$th power of Frobenius followed
by this isomorphism defines
an endomorphism of $G$, which we will denote $\alpha_k \colon G \rightarrow G$.
In particular, $\alpha_k$ is a degree $p^k$ purely inseparable
endomorphism and thus $v(\alpha_k) = k$.
The dual homomorphism of $F$ is separable by
equivalence (ii) of \cite[Thm. V.3.1(a)]{Silverman09}, and so the dual
endomorphism $\overline \alpha_k$ is separable, meaning that $v(\overline
\alpha_k) = 0$. Since $\alpha_k \overline \alpha_k = p^k$ by \cite[Thm.
III.6.2(a)]{Silverman09}, then $(p)$ is not a prime ideal, and
$\overline \alpha_k \not\in m$.

In case (5), the endomorphism $p$ is purely inseparable by the
equivalence (iii) of \cite[Thm. V.3.1(a)]{Silverman09}, and has degree
$p^2$ by \cite[Thm. III.6.2(d)]{Silverman09}. Therefore, $v(p) = 2$. Let
$\alpha$ be any endomorphism in $\EE$, and $\overline \alpha$ its dual.
Then $\alpha \overline \alpha = d$, where $d$ is the degree of $\alpha$
and of $\overline \alpha$, by \cite[Thm. III.6.2(a)]{Silverman09}. If we
write $d = p^k e$, where $p$ does not divide $e$, so that $k$ is the
$p$-adic valuation of $d$, then $v(d) = k v(p) + v(e) \geq 2k$. On the
other hand, the inseparable degree of $\alpha$ divides $d$ and is a
power of $p$, so $v(\alpha) \leq k$, and similarly $v(\overline\alpha)
\leq k$. By the multiplicativity of valuation, $v(\alpha) = k$, which is
what we wanted to show.
\end{proof}

We next want to show that the valuation on $\EE$ is surjective. Since we
already know that the residue division ring of $Q$ is commutative, this
means we can use the results of \S~\ref{sec:Flocks}.

\begin{lm}\label{lm:ValOne}
There exists an element $\pi \in \EE$ with $v(\pi) = 1$.
\end{lm}

\begin{proof}
We consider each of the cases of Proposition~\ref{prop:Class}.
If $G \cong \Ga$, then we can take $\pi = F$. In cases (2) and (3),
$\EE \cong \ZZ$, with the $p$-adic valuation, and in case (4), the
valuation restricts to the $p$-adic valuation on the subring of
integers, and so in these case, we can take $\pi = p$.

In case (5), where $G$ is a supersingular elliptic curve, $v(p) = 2$,
and we have to find $\pi \in \EE \setminus \ZZ$. By
\cite[Thm.~42.1.9]{Voight18}, $Q$ is ramified at $p$, which means that $Q
\otimes_{\QQ} \QQ_p$ is a division algebra over $\QQ_p$, the field of
$p$-adics. Therefore, by \cite[Thm.~13.3.10(c)]{Voight18}, there is an
element $\varphi \in Q \otimes_{\QQ} \QQ_p$ such that $N(\varphi)$ has $p$-adic
valuation 1. Since $Q \otimes_{\QQ} \QQ_p$ is the $p$-adic completion of
$Q$, then $\varphi$ can be approximated by an element $\varphi' \in Q$ with the
same valuation. We can write $\varphi' = a^{-1} \psi$, where $a \in \ZZ$ and
$\psi \in \EE$. Therefore, $N(\psi) = a^2 N(\varphi')$ has odd $p$-adic
valuation, say $2k+1$, so we can write $\psi = \psi' \circ F^{2k+1}$,
where $\psi'\colon G^{(p^{2k+1})} \rightarrow G$ is
separable. Since $G$ is defined over $\FF_{p^2}$ \cite[Thm.
V.3.1(a)]{Silverman09}, $G^{(p^{2k+1})} \cong G^{(p)}$, and so we take
$\pi = \psi' \circ F$.
\end{proof}

\subsection{The derivative homomorphism} \label{ssec:DerHom}
An element $\alpha \in \EE$ is, by definition, an algebraic group
homomorphism $G \to G$. In particular, it maps $0$ to $0$, and the
derivative $d_0 \alpha$ is a linear map from the tangent space $T_0 G$
into itself. Since $T_0 G$ is a one-dimensional vector space over $K$,
we can identify $d_0 \alpha$ with a scalar in $K$.
Concluding, we have a map
\[ \ell:\EE \to K,\quad \alpha \mapsto d_0 \alpha. \]
In the case of an elliptic curve, this is the same map as constructed
in \cite[Corollary III.5.6]{Silverman09}.

\begin{lm}
The map $\ell$ is a (unitary) ring homomorphism, $\im \ell$
is a subfield of $K$, and $\ker \ell$ is the ideal $\{\phi \in \EE \mid
v(\phi) > 0\}$.
\end{lm}

\begin{proof}
First, the multiplicative neutral element of $\EE$ is the identity $G
\to G$, whose derivative is the scalar multiplication $T_0 G \to T_0 G$
by $1 \in K$. Next, multiplicativity of $\ell$ follows from the chain rule:
\[ \ell(\alpha \beta)=d_0(\alpha \circ \beta)=(d_0 \alpha) \circ (d_0 \beta)=
\ell(\alpha)\ell(\beta). \]
For additivity, we recall that for any algebraic group $G$ with neutral
element $e$ the derivative of the group operation $m \colon G \times G
\to G$ at $(e,e)$ is the addition map $T_e G \times T_e G \to T_e G$. So
in our setting where $e=0$,
\[ \ell(\alpha + \beta)=d_0(\alpha+\beta)=d_0(m \circ
(\alpha,\beta))=d_0\alpha+d_0\beta=\ell(\alpha)+\ell(\beta),
\]
where we used the chain rule once more in the third equality.

To show that $\im \ell$ is a field, we use the classification of
one-dimensional groups $G$. If $G$ is $\Ga$, then $\EE$ is the ring of
$p$-polynomials $K[F]$, and $\ell(F) = 0$, so $\im \ell = K[F] / K[F]F$
is isomorphic to $K$. In the other cases, when $G$ is $\Gm$ or an
elliptic curve, $\EE$ is a finitely generated $\ZZ$-algebra. Since $\im
\ell$ is a subring of a field of characteristic $p$, it must be a
finitely generated $\FF_p$-algebra, and also an integral domain, thus it
is a field.

For the last statement, if $\alpha \in \EE$ has positive valuation, then
it is inseparable, so its derivative vanishes. Conversely, let $\alpha
\in \EE$ with $\ell(\alpha)=d_0\alpha=0$. For $h \in G$ let $a_h \colon
G \to G$ be the morphism $g \mapsto g+h$. Then $\alpha \circ a_h =
a_{\alpha(h)} \circ \alpha$ and therefore \[ (d_h \alpha) \circ (d_0
a_h) = d_0 (\alpha \circ a_h) = d_0(a_{\alpha(h)} \circ \alpha) =
d_0(a_{\alpha(h)}) \circ 0 = 0 \] and using that $d_0 a_h$ is invertible
with inverse $d_h a_{-h}$ we find that $d_h \alpha=0$. So $\alpha$ is
inseparable, so it has positive valuation.
\end{proof}

\begin{re}
The first statement in the lemma also holds when $\cha K=0$. Then the
non-existence of inseparable morphisms implies immediately that $\EE$
embeds into $K$ and is, in particular, a commutative ring.
\end{re}

Let $R \subseteq Q$ be the valuation ring. For what follows, we need
to extend $\ell$ from $\EE$ to $R$. Let $\alpha,\beta \in \EE$ such that $\alpha
\beta^{-1} \in R$. By \cite[Corollary II.2.12]{Silverman09}, we can
write $\beta=\beta' \circ F^e$ for some $e \in \ZZo$, where $F^e:G
\to G^{(p^e)}$ is the $e$-th power of Frobenius and $\beta': G^{(p^e)}
\to G$ is separable. It follows that $v(\beta)=e$. Similarly, write
$\alpha=\alpha'' \circ F^d$ with $\alpha'': G^{(p^d)} \to G$ separable and
hence $v(\alpha)=d$. Since $\alpha \beta^{-1} \in R$, we have $d \geq
e$, and we set $\alpha':=\alpha'' \circ F^{d-e}$. Then $\alpha',\beta'$
are both morphisms $G^{(p^e)} \to G$ and $\beta'$ is separable. This
implies that $d_0 \beta':T_0 G^{(p^e)} \to T_0 G$ is multiplication by a
nonzero scalar, and so an isomorphism. Then
$(d_0 \alpha') (d_0 \beta')^{-1}$ is a linear map $T_0 G \to T_0 G$,
hence multiplication by a scalar, which we denote by $\ell(\alpha
\beta^{-1})$.

\begin{lm} \label{lm:ellextends}
The above is a well-defined extension of $\ell\colon \EE \to K$ to a
ring homomorphism $\ell\colon R \to K$ whose image is a field and whose
kernel is the set of elements of positive valuation.
\end{lm}

\begin{proof}
First, taking $\beta$ equal to the identity, the above reduces to the
earlier definition of $\ell\colon \EE \to K$. Second, let $\gamma \in \EE
\setminus \{0\}$ and set $\alpha_1:=\alpha \gamma$ and $\beta_1:=\beta
\gamma$, so that $\alpha_1 \beta_1^{-1} = \alpha \beta^{-1}$. Write
$\gamma=\gamma' F^c$ with $c \in \ZZo$ and $\gamma':G^{(p^c)} \to G$
separable. Then, with notation as above,
\[ \beta_1=(\beta' F^e) (\gamma' F^c) = \beta' (F^e \gamma' F^{-e})
F^{e+c} = \beta' \gamma'' F^{e+c} \]
where $\gamma''=F^e \gamma' F^{-e}$ is a separable morphism $G^{(p^{e+c}))}
\to G^{(p^e)}$. Similarly, we have $\alpha_1=\alpha' \gamma'' F^{e+c}$.
The definition of $\ell(\alpha_1 \beta_1^{-1})$ reads
\[ (d_0 \alpha' \gamma'') (d_0(\beta' \gamma''))^{-1} = (d_0 \alpha')
(d_0 \gamma'') (d_0 \gamma'')^{-1} (d_0 \beta')^{-1} = (d_0 \alpha')
(d_0 \beta')^{-1}. \]
This shows that the definitions of $\ell$ on $\alpha_1 \beta_1^{-1}$
and
on $\alpha \beta^{-1}$ agree. More generally, $\alpha
\beta^{-1}=\alpha_1
\beta_1^{-1}$ holds if and only if there exist nonzero $\gamma,\delta
\in \EE$
such that $\alpha \gamma=\alpha_1 \delta$ and $\beta \gamma=\beta_1
\delta$, and applying the above twice we find that $\ell\colon R \to K$
is well-defined.

Second, to show that $\ell$ is multiplicative, let $r,r_1 \in R$. If
$v(r)>0$
or $v(r_1)>0$, then $v(rr_1)>0$ and $\ell(rr_1)=0=\ell(r)\ell(r_1)$.
So we may assume that $v(r)=v(r_1)=0$. We may also write $r$ and $r_1$
with a common denominator: $r=\alpha \beta^{-1}, r_1=\alpha_1
\beta^{-1}$
where $v(\alpha)=v(\alpha_1)=v(\beta)=:e$. Now find $\gamma, \gamma_1
\in
\EE \setminus \{0\}$ such that $\beta \gamma=\alpha_1 \gamma_1$, so
that
\[ s:=(\alpha \beta^{-1})(\alpha_1 \beta^{-1})=(\alpha \gamma)(\beta
\gamma_1)^{-1} \]
and also $v(\gamma)=v(\gamma_1)=:c$. Write $\gamma=\gamma' F^c$ and
$\gamma_1=\gamma_1' F^c$ and $\alpha=\alpha' F^e$ and
$\alpha_1=\alpha_1'
F^e$ and $\beta=\beta' F^e$ with $\gamma',\gamma_1':G^{(p^c)} \to G$
and $\alpha',\alpha_1',\beta':G^{(p^e)} \to G$ separable. Then we have,
for
the denominator of $s$,
\[ \beta \gamma_1=\beta' F^e \gamma_1' F^c = \beta' \gamma_1'' F^{e+c}
\]
where $\gamma_1''=F^e \gamma_1' F^{-e}:G^{(p^{e+c})} \to G^{(p^e)}$ is
separable.  Similarly, for the numerator of $s$,
\[ \alpha \gamma=\alpha' \gamma'' F^{e+c} \]
where $\gamma''=F^e \gamma' F^{-e}:G^{(p^{e+c})} \to G^{(p^e)}$ is
separable.
Now, by definition,
\begin{equation} \label{eq:ells}
\ell(s)=(d_0 \alpha' \gamma'') (d_0 \beta' \gamma_1'')^{-1}
=(d_0 \alpha') (d_0 \gamma'') (d_0 \gamma_1'')^{-1} (d_0 \beta')^{-1}
\end{equation}
On the other hand, by a similar computation, the relation $\beta
\gamma=\alpha_1 \gamma_1$ implies $\beta' \gamma'' = \alpha_1
\gamma_1''$, so that
\[ (d_0 \beta') (d_0 \gamma'')=(d_0 \alpha_1') (d_0 \gamma_1''). \]
Writing this as $d_0 \gamma''=(d_0 \beta')^{-1} (d_0 \alpha_1') (d_0
\gamma_1'')$ and substituting in~\eqref{eq:ells} yields
\[ \ell(s)=(d_0 \alpha') (d_0 \beta')^{-1} (d_0 \alpha_1') (d_0
\beta')^{-1}=\ell(r) \ell(r_1), \]
as desired.

Third, for additivity of $\ell$ we compute, still assuming $r=\alpha
\beta^{-1},r_1=\alpha_1 \beta^{-1} \in R$ and notation as above, but
no longer requiring $v(r)=v(r_1)=0$,
\[ 
\ell(r+r_1)=\ell((\alpha+\alpha_1) \beta^{-1})
=d_0(\alpha'+\alpha_1') (d_0 \beta)^{-1}
=(d_0 \alpha' + d_0 \alpha_1') (d_0 \beta)^{-1}= \ell(r) + \ell(r_1). 
\]
Finally, $\ker \ell=\{r \in R \mid v(r)>0\}$ follows directly from the
definition. Since every element of $R$ not in this ideal is invertible
in $R$, $\im \ell$ is a field.
\end{proof}

\subsection{The Lie algebra of a subgroup}

We return to the division ring $Q$ generated by the endomorphism ring
$\EE$ of a connected, one-dimensional algebraic group over $K$, equipped
with the valuation from \S \ref{ssec:Valuation}. Note that the residue
field $L$ here is commutative, since by Lemma~\ref{lm:ellextends}, $L$
embeds into the ground field $K$ of our algebraic group $G$. We write
$\ell$ for the map $R^n \to K^n$ defined by applying $\ell$
component-wise. To prove Theorem~\ref{thm:Flock} we need a description
of the Lie algebra of a closed, connected subgroup $X \subseteq G^n$ in
terms of the right vector space representing it.

\begin{lm} \label{lm:LieAlgebra}
Let $X \subseteq G^n$ be a closed, connected subgroup, let $N =P^{-1}(X)
\subseteq \EE^n$ be the saturated right module representing it, and $NQ$
the right subspace of $Q^n$ generated by $Q$. Let $v$ be any vector
spanning the one-dimensional space $T_0 G$.  Then we have
\[ T_0 X=\langle \ell(NQ \cap R^n)v \rangle_K. \]
\end{lm}

\begin{proof}
First, $\dim_K T_0 X=\dim X=\dim_Q NQ$ by Lemma~\ref{lm:P}. On the
other hand, by Lemma~\ref{lm:Reduction} and the fact that $\ell$ is
just the reduction map followed by an embedding $L \to K$, $\dim_Q
NQ=\dim_K \ell(NQ \cap R^n)$. So the two spaces in the lemma have the
same dimension. It therefore suffices to prove that the right-hand side
is contained in the left-hand side.  Since any finite set of elements in 
$Q$ can be given common denominators, a general element of $NQ \cap
R^n$ is of the form $\upsi \beta^{-1}$ with $\beta \in \EE\setminus
\{0\}$ and $\upsi \in N$. Write $\beta=\beta' F^e$ with $e \in \ZZo$
and $\beta':G^{(p^e)} \to G$ separable, and write $\psi_i=\psi_i' F^e$
where $\psi_i':G^{(p^e)} \to G$ is a not necessarily separable morphism.
From $\upsi \in N$ and $\upsi=\upsi' \circ F^e$ and the fact that $F^e:G
\to G^{(p^e)}$ is surjective, it follows that $\upsi'$ maps $G^{(p^e)}$
into $X$. Hence $d_0 \upsi'$ maps $T_0 G^{(p^e)}$ into $T_0 X$.
On the other hand, by definition of $\ell$ we have 
\[ \ell(\upsi \beta^{-1})=(d_0 \upsi') (d_0 \beta')^{-1}, \]
which, therefore, is a linear map $T_0 G \to T_0 X$, as desired.
\end{proof}

\subsection{Frobenius flocks of subgroups}
\label{ssec:Flock}

The last ingredient for Theorem~\ref{thm:Flock} is the notion of Frobenius flock
of a $d$-dimensional, closed, connected subgroup $X$ of $G^n$. By
Lemma~\ref{lm:ValOne} we can choose a uniformizer $\pi$ of $Q$ which is
in fact an element of $\EE$. For $\alpha \in \ZZo^n$ and $q \in G^n$
we write $\pi^\alpha q:=(\pi^{\alpha_1}(q),\ldots,\pi^{\alpha_n}(q))
\in G^n$. This determines an action of the additive monoid $\ZZo^n$
on $G^n$, and $\pi^\alpha X$ is a closed, connected subgroup of $G^n$
for all $\alpha \in \ZZo^n$ of dimension $d$.

We extend the definition of $\pi^\alpha X$ to $\alpha \in \ZZ^n$
as follows. Write $\alpha=\beta - k (1,\ldots,1)$ where $\beta
\in \ZZo^n$ and $k \in \ZZo$. Then we let $\pi^\alpha X$ be the
connected component of $0$ of the preimage of $\pi^\beta X$ under
the homomorphism $\pi^{k(1,\ldots,1)}:G^n \to G^n$. This preimage is
clearly contained in the preimage of $\pi^{\beta+(1,\ldots,1)}X$ under
$\pi^{(k+1)(1,\ldots,1)}$ and hence, since both preimages have dimension
$d$, the connected components of $0$ coincide, so $\pi^\alpha X$ is
independent of the choice of $\beta$ as above. One readily checks that
we have thus defined an action of $\ZZ^n$ on the set of $d$-dimensional
closed, connected subgroups of $G^n$.

\begin{de}
The {\em Frobenius flock} of $X \subseteq G^n$, relative to $\pi$, is
the collection of vector spaces $(U_\alpha)_{\alpha \in \ZZ^n}$ defined
by $U_\alpha:=T_0 (\pi^{-\alpha} X) \subseteq (T_0G)^n$.
\end{de}

In the case where $G=\Ga$ and $\pi=F$, this notion coincides with
the Frobenius flocks of \cite{Bollen17}.  There, $X$ was allowed to
be an arbitrary irreducible closed subset of $K^n$, and consequently
the base point at which we took the tangent spaces had to be chosen somewhat
carefully. In our current setting, where $X$ is a subgroup, all points
lead to equivalent Frobenius flocks, which is why we chose $0$ as the
base point. Also, in order to work with elliptic curves for which
Frobenius is not an endomorphism, we allow $\pi$ to be any endomorphism
of valuation 1.

\subsection{Proof of Theorem~\ref{thm:Flock}.} \label{ssec:ProofFlock}

The next proposition says that the Frobenius flock of $X$ equals
the linear flock of the corresponding right $Q$-subspace of $Q^n$,
up to a natural identification. Recall that the residue field $L$
of $Q$ is a subfield of $K$ via the homomorphism $\ell$ from
\S \ref{ssec:DerHom}. 

\begin{prop}
Let $X \subseteq G^n$ be a closed, connected subgroup and $N:=P^{-1}(X)
\subseteq \EE^n$ the corresponding right module. Let $v \in T_0 G
\setminus \{0\}$. Let $(V_\alpha)_\alpha$ be the $L$-linear flock of
$NQ$ and let $(U_\alpha)_\alpha$ be the Frobenius flock of $X$.  Then the
map $L^n \to (T_0 G)^n, c=(c_1,\ldots,c_n) \mapsto (\ell(c_1)
v,\ldots, \ell(c_n) v)=\ell(c)v$ induces a linear bijection 
$(K \otimes_L V_\alpha) \to U_\alpha$ for each $\alpha \in \ZZ^n$.
\end{prop}

\begin{proof}
Let $\alpha \in \ZZ^n$ and set $Y:=\pi^{-\alpha} X$. A straightforward
computation shows that $P^{-1}(Y)Q=\pi^{-\alpha}NQ$. By
Lemma~\ref{lm:LieAlgebra} applied to $Y$, we therefore have
\[ U_\alpha = T_0 Y=\langle \ell((\pi^{-\alpha}NQ) \cap R^n)v \rangle_K
= (K \otimes_L \overline{(\pi^{-\alpha} NQ) \cap R^n}) v=(K \otimes_L
V_\alpha)v, \]
as desired.
\end{proof}

In what follows, we decompose $\pi=\psi \circ F$ for some separable
homomorphism $\psi\colon G^{(p)} \to G$ and $F$ the Frobenius map. Let
$h\colon G \dto K$ be a rational function defined near $0$ with $h(0)=0$
and such that $d_0 h\colon T_0 G \to T_0 K=K$ is nonzero, hence an
isomorphism. Let $h^{(p)}\colon G^{(p)} \dto K$ be the Frobenius twist
of $h$, i.e., the rational map making the diagram on the left commute:

\begin{center}
\includegraphics{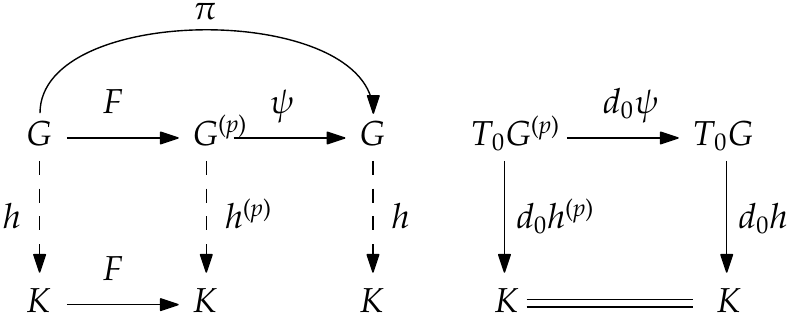}
\end{center}

We want the diagram on the right, at the level of tangent spaces, to
commute as well. {\em A priori}, $d_0 h^{(p)}=c d_0 h \circ d_0 \psi$
for some constant $c \in K^*$. Multiplying $h$ by a scalar $a \in K$,
the differential $d_0 h^{(p)}$ on the left side of the equation is
multiplied by $a^p$ and, on the right, $d_0 h$ is
multiplied by $a$. Hence if we choose $a$ such that $a^{p-1}=1/c$,
then we have the desired equality:
\[ d_0 h \circ d_0 \psi=d_0 h^{(p)}. \]
The mild conditions alluded to in Theorem~\ref{thm:Flock} are
as follows:
\begin{equation} \tag{**}
\text{\em $h:G \dto K$ is defined near $0$, $d_0 h \neq 0$, and
$d_0 h \circ d_0 \psi = d_0 h^{(p)}$.}
\end{equation}
Applying the Frobenius twist to both sides, we obtain 
\[ d_0 h^{(p)} \circ d_0 \psi^{(p)}=d_0 h^{(p^2)} \]
where $\psi^{(p)}:G^{(p^2)} \to G^{(p)}$. Combining the two formulas
yields
\[ d_0 h \circ d_0 \psi \circ d_0 \psi^{(p)} = d_0 h^{(p^2)}.
\]
We abbreviate $\psi \circ \psi^{(p)}$ to $\psi^2$. Then the
above reads 
\[ d_0 h \circ d_0 \psi^2 = d_0 h^{(p^2)}. \]
More generally, for each nonnegative integer $k$, writing 
$\psi^k:=\psi \circ \psi^{(p)} \circ \cdots \circ \psi^{(p^{k-1})}$, we have 
\[ d_0 h \circ d_0 \psi^k=d_0 h^{(p^k)}. \]
Extending this component-wise to tuples we have 
\[ d_0 h^n \circ d_0 \psi^\alpha = d_0 h^{(p^\alpha)} \text{ for all }
\alpha \in \ZZo^n. \]

\begin{prop} \label{prop:FlockBDP}
Set $Y:=\overline{h^n(X)}$ and assume that $Y^{(p^\alpha)}$ is smooth
at $0$ for all $\alpha \in \ZZ^n$.  Let $(U_\alpha)_\alpha$ be the
Frobenius flock of $X$, and for $\alpha \in \ZZ^n$ set $W_\alpha:=T_0
Y^{(p^\alpha)}$. Then the matroid on $[n]$ defined by $W_\alpha$ equals
the matroid defined by $T_{F^\alpha y} Y^{(p^\alpha)}$ for $y$ a general
point in $Y$, and the Frobenius flock $(W_\alpha)_\alpha$
of $Y$ at $0$ is the image of $(U_\alpha)_\alpha$ under the linear
isomorphism $(d_0 h)^n: T_0 G^n \to T_0 K^n$.
\end{prop}

\begin{proof}
Since Frobenius flocks are uniquely determined by their restriction
to the positive orthant, it suffices to prove both statements for
$\alpha \in \ZZo^n$. For the last statement, consider the following
two diagrams; here we have left out the obvious solid arrows between
$X,X^{(p^\alpha)},\pi^\alpha X$ and between $Y,Y^{(p^\alpha)}$ as well
as the dashed arrows $h^n: X \dto Y$ and $h^{(p^\alpha)}: X^{(p^\alpha)}
\dto Y^{(p^\alpha)}$.

\begin{center}
\includegraphics{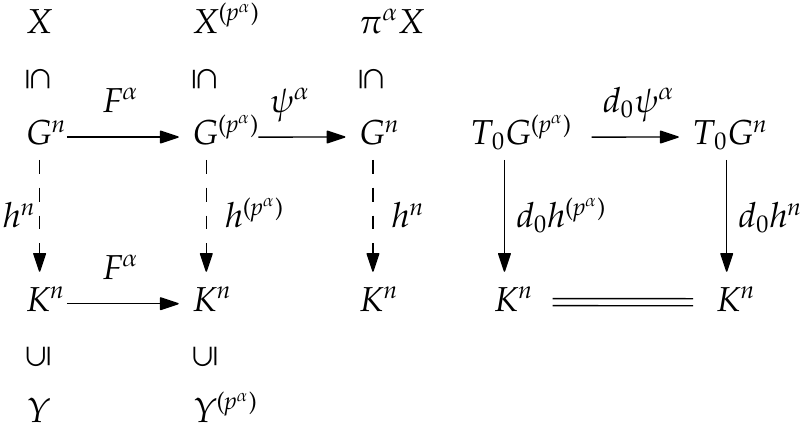}
\end{center}

The left-most diagram commutes by definition, and the right-most
diagram commutes by the discussion above. So, by the left-most diagram,
$A:=d_0 h^{(p^\alpha)} \circ (d_0 \psi^\alpha)^{-1}$ maps $U_\alpha=T_0
\pi^\alpha X$ into $W_\alpha=T_0 Y^{(p^\alpha)}$, and by the
right-most
diagram $A$ equals $d_0 h^n$. This proves the last
statement.

\begin{figure}
\begin{center}
\includegraphics{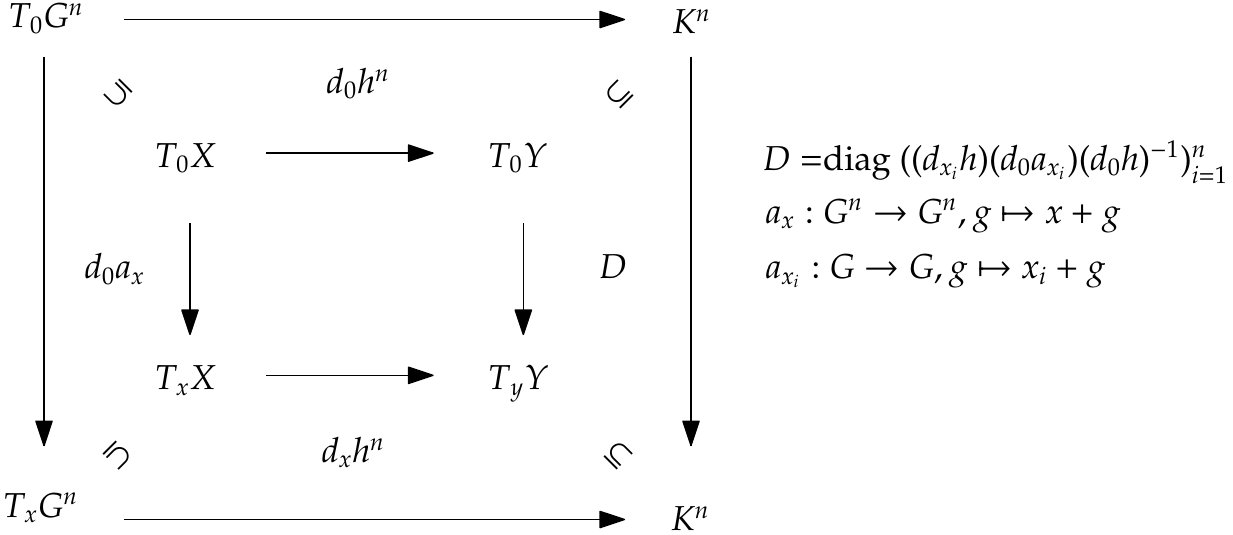}
\caption{The last commuting diagram in the proof of
Proposition~\ref{prop:FlockBDP}.} \label{fig:Diag}
\end{center}
\end{figure}

For the first statement, we will use the homogeneity of $X$.  A general
point $y$ in $Y$ has the same properties as $0$: $y$ is of the form
$h^n(x)$ for some $x=(x_1,\ldots,x_n) \in X$, $d_{x_i} h$ is nonzero for
each $i$, and $d_x h^n$ is an isomorphism between $T_x X$ and $T_y Y$. In
that case, we have the commuting diagram of
Figure~\ref{fig:Diag}, where all linear
maps are isomorphisms and the right-most vertical map has a diagonal
matrix relative to the standard basis. 
Consequently, the matroids represented by $T_{0}Y$ and $T_{y} Y$
are equal.
\end{proof}

\section{Equivalence of algebraic representations}
\label{sec:Equivalence}

In our examples below we will need to characterize the algebraic
representations of certain matroids, up to equivalence. Here we discuss
briefly what this means. For this we go back to the original definition
of algebraic matroids: let $x_1,\ldots,x_n$ be elements of an extension
field $L$ of our algebraically closed ground field $K$. In the matroid
$M$ on $[n]$ determined by this data, $I$ is independent if and only if
the $x_i,i \in I$ are algebraically independent over $K$. In the set-up
of the introduction, $L$ is the function field of $X \subseteq K^n$ and
the $x_i$ are the coordinate functions. Clearly, if we enlarge $L$, the
matroid remains the same, and if we add $x_{1}', \ldots, x_{n}' \in L$
such that the algebraic closure of $K(x_i)$ in $L$ equals that of
$K(x_{i}')$ for each $i$, then the corresponding elements $i$ and $i'$
are parallel.
Thus, the matroid of the restriction to $x_{1}', \ldots,
x_{n}'$ is again $M$, and we call this algebraic realization
\emph{equivalent} to the original one.

In this section we will show that the Lindstr\"om valuations of equivalent
algebraic representations differ only by a trivial valuation.  This
follows from a more general statement about valuations of matroids with
parallel elements.  We use the following lemma, which is a straightforward
consequence of submodularity of matroid valuations.

\begin{lm}
\label{lm:valchar}
Let $\nu: \binom{E}{r} \rightarrow \RR \cup \{\infty\}$ be a matroid
valuation.  Let $S \in \binom{E}{r-2}$ and $\{a,b,c,d\} \in \binom{E
\setminus S}{4}$ be given.  Then the minimum of
\begin{eqnarray*}
\nu(S \cup \{a,b\}) &+& \nu(S \cup \{c,d\}),\\
\nu(S \cup \{a,c\}) &+& \nu(S \cup \{b,d\}), \textrm{ and}\\
\nu(S \cup \{a,d\}) &+& \nu(S \cup \{b,c\})
\end{eqnarray*}
is attained at least twice. \hfill $\square$
\end{lm}

\begin{prop} \label{prop:Parallel}
Let $M$ be a matroid on $E$, and suppose $i,j \in E$ are parallel in $M$.
Let $\nu$ be a valuation of $M$.
Then there exists $c_{i,j} \in \RR$ such that
\[\nu(S \cup \{i\}) - \nu(S \cup \{j\}) = c_{i,j}\]
for all $S \subseteq E \setminus \{i,j\}$ for which $S \cup \{i\}$
(and hence also $S \cup \{j\}$) is a basis of $M$.
\end{prop}

\begin{proof}
Suppose $S,S'$ are such that
\[\nu(S \cup \{i\}) - \nu(S \cup \{j\}) \neq \nu(S' \cup \{i\}) - \nu(S' \cup \{j\}),\]
and $|S \setminus S'|$ is minimal.
Clearly $S \neq S'$.
Using the basis exchange axiom of matroids, let $a \in S \setminus S'$ and $b \in S' \setminus S$ be given such that $S \setminus \{a\} \cup \{b,i\}$ is a basis.
Then
\[\nu(S \setminus \{a\} \cup \{b,i\}) - \nu(S \setminus \{a\} \cup \{b,j\}) = \nu(S' \cup \{i\}) - \nu(S' \cup \{j\})\] by minimality of $|S \setminus S'|$.
Since $i$ and $j$ are parallel, $\nu(S \setminus \{a\} \cup \{i,j\}) = \infty$.
By Lemma \ref{lm:valchar}, we have
\[\nu(S \cup \{i\}) - \nu(S \cup \{j\}) = \nu(S \setminus \{a\} \cup
\{b,i\}) - \nu(S \setminus \{a\} \cup \{b,j\}),\] which is a contradiction.
\end{proof}

\begin{prop} \label{prop:Equivalence}
Let $L \supseteq K$ be a field extension and let
$x_1,\ldots,x_n,x_1',\ldots,x_n' \in L$ be elements. Suppose that, for
each $i$, the
algebraic closure of $K(x_i)$ in $L$ is the same as that of $K(x_i')$,
i.e., that the elements $(x_1,\ldots,x_n)$ and $(x_1',\ldots,x_n')$
determine equivalent algebraic representations of the same matroid $M$.
Let $\mu,\mu'$ be the Lindstr\"om valuations of these representations.
Then there exists an $\alpha \in \RR^n$ such that
\[ \mu'(B)=\mu(B) + e_B^T \alpha \text{ for all bases $B$ of $M$.} \]
\end{prop}

\begin{proof}
The tuple $(x_1,\ldots,x_n,x_1',\ldots,x_n')$ is an algebraic
representation of the matroid obtained from $M$ by adding a parallel
copy $i'$ to each element $i$. Denote the Lindstr\"om valuation of this
algebraic representation by $\nu$. From \cite[Theorem 1]{Cartwright18}
and multiplicativity of inseparable degree \cite[Cor. V.6.4]{Lang02}, it
follows that $\nu$ restricts to $\mu$ on the copy of $[n]$ corresponding
to $x_1,\ldots,x_n$ and to $\mu'$ on the copy of $[n]$ corresponding
to $x_1',\ldots,x_n'$. Furthermore, by Proposition~\ref{prop:Parallel},
for any basis $\{i_1, \ldots, i_r\}$ of $M$, we have
\[\nu(\{i_1, \ldots, i_r\}) = \nu(\{i_1', \ldots, i_r'\}) + c_{i_1, i_1'} + \ldots + c_{i_r, i_r'}.\]
Hence $\mu$ and $\mu'$ differ by the trivial valuation $B \mapsto e_B^T
\alpha$ with $\alpha_i = c_{i,i'}$ for each $i$.
\end{proof}

\section{Examples} \label{sec:Examples}

\begin{ex}
Consider the matroid $M$ from Figure \ref{fig:ellipticexample}.
\begin{center}
\begin{figure}
\includegraphics{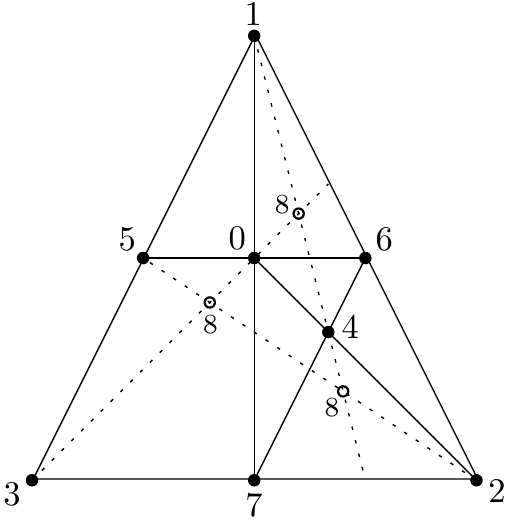}
\caption{A matroid of rank 3 on 9 elements. A triple is collinear if and only if it is dependent in the matroid. The point 8 is the common intersection of the lines through $\{0,3\}$, $\{1,4\}$ and $\{2,5\}$.}  \label{fig:ellipticexample}
\end{figure}
\end{center}

We construct a general matrix over a division ring $S$ such that each dependent set of $M$ is dependent in the matrix:
\[\Psi = \bordermatrix{ & & & \cr
   0& 1 & 0 & 0 \cr
   1& 0 & 1 & 0 \cr
   2& 0 & 0 & 1 \cr
   3& 1 & 1 & a \cr
   4& 1 & 0 & -1 \cr
   5& 1 & a & a \cr
   6& 0 & 1 & 1 \cr
   7& 1 & 1 & 0 \cr
   8& 1 & a & -1},
  \]
where $a \in K$ satisfies $a^2 = -1$.
After choosing a basis and fixing row and column scalars, the remaining entries were chosen as freely as possible given the dependent triples of $M$.
As it turns out, the only freedom that is left is the choice of $a$.

No assumptions on the characteristic or commutativity of $S$ have been
made at this point.  If $S$ has characteristic 2, then (among others)
the rows corresponding to the basis $\{3,4,5\}$ become dependent, so that
$\Psi$ cannot be a representation of $M$.  So $S$ must have characteristic
$\neq 2$. If, for example, $S = \QQ(i)$ and $a = i$, then a subset
of rows of $\Psi$ is dependent if and only if it is dependent in $M$.
Hence the column space of $\Psi$ is a representation of $M$ over $\QQ(i)$.

Now suppose that $K$ is a field of characteristic 2, and $G$ is a
connected one-dimensional algebraic group over $K$.  Let $X$ be a closed,
connected subgroup of $G^n$, representing $M$ algebraically.  Then by
Theorem \ref{thm:bijections}, there is a linear representation of $M$
over the endomorphism ring $\EE$ of $G$.  Due to the above, $M$ is not
representable over $\QQ$, nor over $K(F)$.  Hence $G$ cannot be either
the additive or multiplicative group.  So $G$ must be an elliptic curve.

And indeed, for the supersingular elliptic curve $G$
from Example~\ref{ex:curve}, $\EE$ is isomorphic to the Hurwitz
quaternions.  So by taking $a \in \EE$ the element corresponding to
$i$, we find that $M$ is a matroid over the one-dimensional group
$G$. To compute the Lindstr\"om valuation of this realization
of $M$, we invoke Theorem~\ref{thm:Flock} and Proposition
\ref{prop:Class}(\ref{item:quaternion}), which says that the valuation
$v$ on $\EE$ maps $\alpha$ to the $2$-adic valuation of $\alpha
\bar{\alpha}$. Then $B \mapsto v(\det \Psi_B)$ is the corresponding
Lindstr\"om valuation of $M$.

\hfill $\clubsuit$
\end{ex}

\begin{ex}
Let $M$ be the non-Fano matroid. As is well-known, $M$ is realizable
over a division ring if and only if the characteristic is not $2$, in
which case there is a projectively unique realization, given by the
rows of the matrix
\[
\begin{pmatrix}
1 & 0 & 0 \\
0 & 1 & 0 \\
0 & 0 & 1 \\
1 & 1 & 0 \\
1 & 0 & 1 \\
0 & 1 & 1 \\
1 & 1 & 1
\end{pmatrix}.
\]
Nonetheless, $M$ has algebraic realizations over a field of
characteristic 2 using either the group $\Gm$ or an elliptic curve, both
of whose endomorphism rings $\EE$ have characteristic $0$.

If $w$ denotes the Lindstr\"om valuation of one of these
algebraic realizations, then, using Theorem~\ref{thm:Flock} and
Proposition~\ref{prop:Determinant}, we can compute the following
invariant:
\begin{align*}
\gamma(w) &:=
w(\{4,5,6\}) - w(\{1,2,5\}) - w(\{1,3,6\}) - w(\{2,3,4\}) + 2
w(\{1,2,3\}) \\
&= v(-2) - v(1) - v(-1) - v(1) + 2v(1) = v(2),
\end{align*}
where $v$ is the valuation on $\EE$. For each element of $M$, the
sum of the coefficients of the valuations in which it appears in the
definition of $\gamma(w)$ is $0$, and so $\gamma$ is invariant under
adding trivial valuations. By the classification of the
valuations in Proposition~\ref{prop:Class}, $\gamma(w)$ is either $2$ if
the algebraic group $G$ is a supersingular elliptic curve or $1$
otherwise.

Moreover, we claim that by results in \cite{Evans91}, every algebraic
realization of $M$ is equivalent to an $\EE$-linear realization, and so
by Proposition~\ref{prop:Equivalence}, $\gamma(w)$ is either 1 or 2 for
the Lindstr\"om valuation~$w$ on any algebraic realization of~$M$. Indeed,
\cite[Thm.\ 2.1.2]{Evans91} states that any algebraic realization of the
matroid $M(K_4)$ of the complete graph on 4 vertices is equivalent to
an $\EE$-linear realization. The deletion of the 6th element of $M$ is
isomorphic to $M(K_4)$, and therefore in any algebraic realization of
$M$, the elements other than the 6th are equivalent to an $\EE$-linear
realization. More specifically, if $L$ is an algebraically closed field
extension of
$K$ and $x_1,\ldots,x_7 \in L$ have the algebraic dependencies over $K$
specified by $M$, then there exists a connected, one-dimensional algebraic
group $G$, a connected subgroup $X \subseteq G^7$, and an embedding $K(X)
\to L$ that, for each $i \in \{1,\ldots,5,7\}$, maps the $i$-th coordinate
function to an element $x_i' \in L$ such that the algebraic closures
$\overline{K(x_i)}$ and $\overline{K(x_i')}$ in $L$ coincide. Now let
$x_6' \in L$ be the image of the 6th coordinate. The fact that the
rank-two flats in $M$ spanned by $1,7$ and by $2,3$ intersect precisely in $7$
implies that
\[ \overline{K(x_6')}
=\overline{K(x_1',x_7')} \cap \overline{K(x_2',x_3')}=
\overline{K(x_1,x_7)} \cap \overline{K(x_2,x_3)}=\overline{K(x_6)}; 
\] 
this proves the claim. 

In conclusion, scaling the valuation on $M$ coming from the
$\Gm$-realization by a positive integer yields infinitely many
valuations, of which exactly two are realizable as the Lindstr\"om
valuation of an algebraic realization over a field of characteristic 2.
In contrast, all of these scaled valuations are realizable by Frobenius
flocks over $\mathbb F_2$ by scaling the original Frobenius flock.
\hfill $\clubsuit$
\end{ex}

We can now prove the last of the results from the introduction. We refer
to \cite{Dress92} for contractions, restrictions, and duality of matroid
valuations. 

\begin{proof}[Proof of Theorem~\ref{thm:NonDual}]
We use the universality construction in \cite[Lem.~3.4.1]{Evans91} to
construct the dual matroid $M^*$, from the following system of equations
in $p^3$ non-commuting variables $u, y_0, \ldots, y_{p^3-2}$:
\begin{align}
y_1 &= u y_0 \notag \\
y_2 &= u y_1 \notag \\
&\vdots \label{eq:NonDual} \\
y_{p^3-2} &= uy_{p^3-3} \notag \\
y_0 &= u y_{p^3-2} \notag \\
y_p &= y_0 u \notag
\end{align}
Then, by \cite[Lem.\ 3.4.1]{Evans91}, there exists a matroid $M^*$ such
that any algebraic realization of $M^*$ is equivalent to a realization by
a closed, connected subgroup $Y \subset G^{n}$ for some one-dimensional
algebraic group $G$. Moreover, setting $\EE:=\End(G)$ and letting $Q$
be the division ring generated by $\EE$, that subgroup realization
corresponds to an assignment of distinct values from $Q$ to the variables
$u, y_0, \ldots, y_{p^3-2}$ such that the equations (\ref{eq:NonDual}) are
satisfied; and any such choice yields a subgroup realization. We now
assume that we have such a realization and study the solutions in an
endomorphism ring $\EE$ of some one-dimensional group.

In particular, we have $y_{i} = u^iy_0$ for all $i$, and, by the second
to last equation, $y_0 =
u y_{p^3-2} = u^{p^3-1} y_0$, 
 which means that $u$ is a $(p^3-1)$-root
of unity, and since the $y_i$ are distinct, $u$ is a primitive
$(p^3-1)$-root of unity. Over $\QQ$, primitive $(p^3-1)$-roots of
unity have degree at least $6$, but all elements in the endomorphism
rings of $\Gm$ and of elliptic curves have degree at most $2$ over
$\QQ$. Therefore, any $\EE$-linear realization of $M^*$ comes from $G =
\Ga$ and $\EE = K[F]$, and $u$ is an element of $\FF_{p^3} \setminus
\FF_p$. In addition, the last equation of~(\ref{eq:NonDual}) gives $y_0
u = y_p = u^p y_0$.

Now, consider the
semidirect product $K^* \rtimes \ZZ$, where the generator of $\ZZ$ acts
on $K^*$ by the
Frobenius automorphism, and let $\mu$ denote
the homomorphism from the monoid $(K[F] \setminus \{0\}, \cdot)$ to
$K^* \rtimes \ZZ$ defined by
\begin{equation*}
\mu\Bigg(\sum_{i=0}^d a_i F^i\Bigg) = (a_d, d),
\end{equation*}
where $d$ is chosen so that $a_d$ is non-zero. Note that the valuation
on $K[F]$ is the projection of $\mu$ onto the second coordinate. Since
$K^* \rtimes \ZZ$ is a group, $\mu$ extends uniquely to a group
homomorphism $\mu\colon Q^* \rightarrow K^* \rtimes \ZZ$, where $Q$ is
the division ring generated by $\EE = K[F]$. Applying $\mu$ to the equation
$y_0 u = u^p y_0$, and using $(a,d)$ to denote
$\mu(y_0)$, we have:
\begin{align*}
(a, d)\cdot (u, 0) &= (u^p, 0) \cdot (a, d) \\
(a u^{p^d}, d) &= (u^p a, d)
\end{align*}
Since $K^*$ is commutative, $u^{p^d} = u^p$. Also, $u$ is in $\FF_{p^3}
\setminus \FF_p$, so the action of Frobenius on $u$ has order $3$, and
so this means that $d = v(y_0) \equiv 1 \bmod 3$.

The construction of $M^*$ in \cite[Lem. 3.4.1]{Evans91} represents the
variable $y_0$ as a cross-ratio, meaning that there are 4 elements
in $M^*$ (denoted $x_0$, $x_\infty$, $x_1$, and $y_0$ in the proof there
but denoted $1,2,3,4$ here)
such that any $Q$-realization of $M^*$ is equivalent to one whose
restriction to these 4 elements is:
\begin{equation*}
\begin{pmatrix}
1 & 0 \\
0 & 1 \\
1 & 1 \\
1 & y_0
\end{pmatrix},
\end{equation*}
whose matroid is the uniform matroid $U_{2,4}$.
By Proposition~\ref{prop:Determinant}, the restriction of the Lindstr\"om
valuation to this $U_{2,4}$, denoted $w^*$, satisfies:
\begin{equation}\label{eq:ValU24}
\begin{split}
w^*(\{1,4\}) + w^*(\{2,3\}) - w^*(\{1, 3\}) -
w^*(\{2,4\}) = \\
v(y_0) + v(-1) - v(1) - v(-1) = v(y_0),
\end{split}
\end{equation}
as does any valuation that differs from $w^*$ by a trivial valuation. 

We now construct a realization of the matroid $M$ dual to $M^*$. First,
we can construct a $K[F]$-realization of $M^*$ by choosing $u$ to be
an element of $\FF_{p^3} \setminus \FF_p$, $y_i = u^{i} F$ for $i = 1,
\ldots, p^3-2$.
Then, by Theorem~\ref{thm:duality}, $M$ also
has a $K[F]$-realization, which is constructed by taking the orthogonal
complement of the realization of $M^*$---this orthogonal complement is
a left $Q$-vector space---and applying the anti-isomorphism $\tau$ which sends $F$ to
$F^{-1}$. In particular, the contraction of all of $M$ except the
elements $1,\ldots,4$ from the previous paragraph
has the realization:
\begin{equation*}
\begin{pmatrix}
-1 & -1 \\
-1 & -F^{-1} \\
1 & 0 \\
0 & 1
\end{pmatrix}.
\end{equation*}
Again, using Proposition~\ref{prop:Determinant} to compute the valuation
$w$ on this realization, we have:
\begin{equation*}
w(\{1,4\}) + w(\{2,3\})
- w(\{1,3\}) - w(\{2,4\})
= v(-1) + v(F^{-1}) - v(1) - v(-1) = -1
\end{equation*}
Therefore, if $w^*$ is the dual valuation of $w$, defined by $w^*(B) =
w(\{1,2,3,4\} \setminus B)$, then it satisfies:
\begin{equation*}
w^*(\{2,3\}) + w^*(\{1,4\})
- w^*(\{2,4\}) - w^*(\{1,3\}) = -1.
\end{equation*}
This valuation is not the Lindstr\"om valuation of an algebraic
realization by \eqref{eq:ValU24} and the requirement that $v(y_0) \equiv 1
\bmod 3$.
\end{proof}

\begin{ex}
We take $G = \Ga$ over any field of positive characteristic $p$, and let
$\lambda$ be any element of $K \setminus \FF_p$. Then, $N$ will be given by the matrix
\begin{equation*}
\bordermatrix{ & x & y & z & u \cr
1  & 1 & 1 & 1 & 1 \cr
2  & 1 & 0 & 1 & 0 \cr
3  & 0 & 1 & 0 & 1 \cr
4  & F & 1 & F & 1 \cr
5  & 1 & 1 & 0 & 0 \cr
6  & 1 & 0 & 0 & 0 \cr
7  & 0 & 1 & 0 & 0 \cr
8  & F & 1 & 0 & 0 \cr
9  & 0 & 0 & 1 & 1 \cr
10 & 0 & 0 & 1 & 0 \cr
11 & 0 & 0 & 0 & 1 \cr
12 & 0 & 0 & F & 1 \cr
13 & \lambda & \lambda & 1 & 1 \cr
14 & \lambda & 0 & 1 & 0 \cr
15 & 0 & \lambda & 0 & 1 \cr
16 & \lambda F & \lambda & F & 1 \cr
17 & F \lambda & \lambda & F & 1
}
\end{equation*}
\begin{figure}
\includegraphics{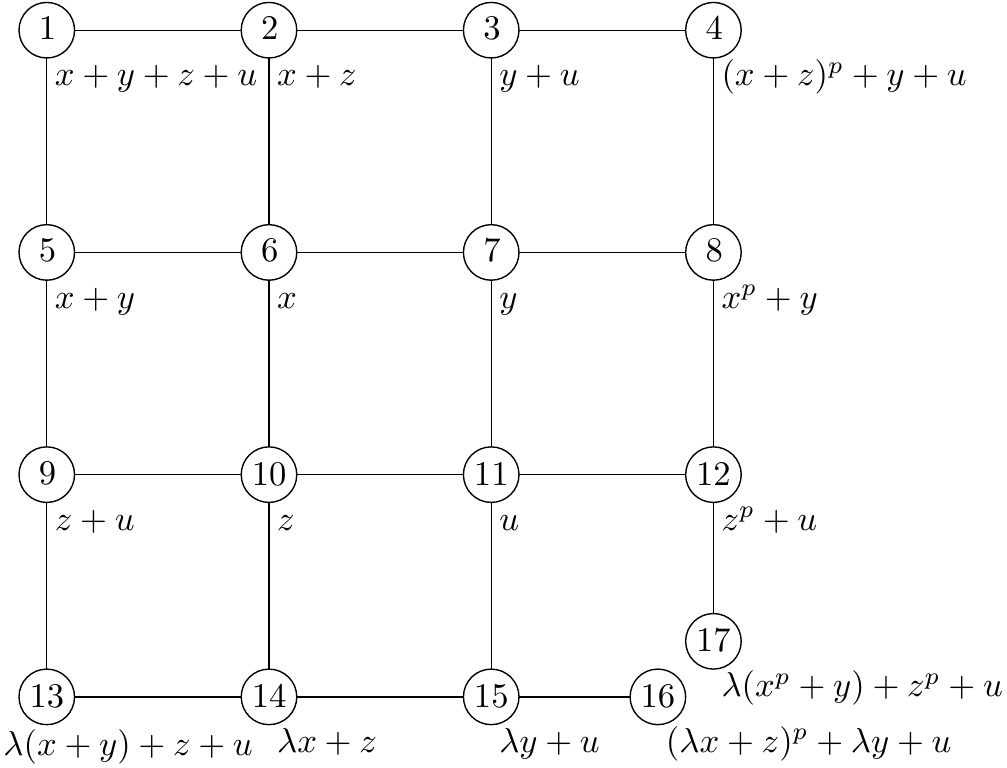}
\caption{An algebraic matroid in $K(x,y,z,u)$ \cite{Lindstrom86a}.}\label{fig:non-comm-example}
\end{figure}%
The first fifteen rows of this matrix can be obtained by interpreting
the field elements in \cite[Fig.~1]{Lindstrom86a} as group homomorphisms
from $\Ga^4$ to $\Ga$. In addition, $N$ arises as a sort of
Kronecker product of the matrices
\begin{equation*}
N_1 = \begin{pmatrix}
1 & 1 \\
1 & 0 \\
0 & 1 \\
\lambda & 1 
\end{pmatrix}
\qquad\mbox{and}\qquad
N_2 = \begin{pmatrix}
1 & 1 \\ 
1 & 0 \\
0 & 1 \\
F & 1
\end{pmatrix},
\end{equation*}
with the bottom row duplicated to have both possible ways of multiplying
the bottom left elements $\lambda$ and $F$ of the matrices.

Then the matroid~$M$ for $N$ is depicted in
Figure~\ref{fig:non-comm-example}, where the lines denote the rank 2
flats, each of which arise from fixing one row in either $N_1$ or $N_2$.
The matroid $M$ is not linear over any field, but is linear over any
non-commutative division ring. Indeed, the matrix $N$ defines the same
matroid if $\lambda$ and $F$ are any pair of non-commuting elements in
any
division ring, which they are in the endomorphism ring of~$\Ga$, since we
required that $\lambda^p \neq \lambda$.
This implies, for instance, that
the submatrix in rows $8,12,14,15$ has full rank, which it would not have
if $\lambda$ and $F$ commuted. Since $M$ is linear over the endomorphism
ring of $\Ga$ in any positive characteristic, it is algebraic over any
field of positive characteristic. Many restrictions of $M$ are also not
linear over any field, such as the restriction to $\{1, 4, 7, 8, 10, 11,
13, 14\}$, which appears in \cite{Ingleton71} and \cite{Lindstrom86a}.
The restriction to $\{3, 4, 7, 8, 9, 10, 13, 14\}$ appears in
\cite[Fig.~3]{Lindstrom85}, where it is incorrectly said not to be
realizable over a division ring, but an algebraic realization over $\FF_2$
is nonetheless given. \hfill $\clubsuit$
\end{ex}

\bibliographystyle{alpha}
\bibliography{diffeq,draismajournal}

\begin{thebibliography}{{Coh}95}

\bibitem[BB19]{Baker19}
Matthew {Baker} and Nathan {Bowler}.
\newblock {Matroids over partial hyperstructures.}
\newblock {\em {Adv. Math.}}, 343:821--863, 2019.

\bibitem[BDP18]{Bollen17}
Guus~P. Bollen, Jan Draisma, and Rudi Pendavingh.
\newblock Algebraic matroids and {F}robenius flocks.
\newblock {\em Adv.~Math.}, 323:688--719, 2018.

\bibitem[Bol18]{Bollen18}
Guus~Pieter Bollen.
\newblock {\em Frobenius {Flocks} and {Algebraicity} of {Matroids}}.
\newblock PhD thesis, Eindhoven University of Technology, 2018.
\newblock Available online at \verb+https://research.tue.nl/+.

\bibitem[{Bor}91]{Borel91}
Armand {Borel}.
\newblock {\em {Linear algebraic groups. 2nd enlarged ed.}}
\newblock New York etc.: Springer-Verlag, 2nd enlarged ed. edition, 1991.

\bibitem[Car18]{Cartwright18}
Dustin Cartwright.
\newblock Construction of the {L}indstr\"om valuation of an algebraic
  extension.
\newblock {\em J. Comb. Theory, Ser. A}, 157:389--401, 2018.

\bibitem[CGP10]{Conrad10}
Brian {Conrad}, Ofer {Gabber}, and Gopal {Prasad}.
\newblock {\em {Pseudo-reductive groups.}}
\newblock Cambridge: Cambridge University Press, 2010.

\bibitem[{Coh}95]{Cohn95}
P.~M. {Cohn}.
\newblock {\em {Skew fields. Theory of general division rings.}}, volume~57.
\newblock Cambridge: Cambridge Univ. Press, 1995.

\bibitem[DD18]{DAli18}
Alessio D'Al\`i and Emanuele Delucchi.
\newblock {S}tanley-{R}eisner rings for symmetric simplicial complexes,
  g-semimatroids and abelian arrangements.
\newblock 2018.
\newblock Preprint, \verb+arXiv:1804.07366+.

\bibitem[DM13]{Adderio13}
Michele {D'Adderio} and Luca {Moci}.
\newblock {Arithmetic matroids, the {T}utte polynomial and toric arrangements.}
\newblock {\em {Adv. Math.}}, 232(1):335--367, 2013.

\bibitem[DP05]{DeConcini05}
C.~{De Concini} and C.~{Procesi}.
\newblock {On the geometry of toric arrangements.}
\newblock {\em {Transform. Groups}}, 10(3-4):387--422, 2005.

\bibitem[DP19]{Delucchi19}
Emanuele Delucchi and Roberto Pagaria.
\newblock The homotopy type of elliptic arrangements.
\newblock 2019.
\newblock Preprint, \verb+1911.02905+.

\bibitem[DW92]{Dress92}
Andreas W.~M. {Dress} and Walter {Wenzel}.
\newblock {Valuated matroids.}
\newblock {\em {Adv. Math.}}, 93(2):214--250, 1992.

\bibitem[EH91]{Evans91}
David~M. Evans and Ehud Hrushovski.
\newblock Projective planes in algebraically closed fields.
\newblock {\em Proc. London Math. Soc.}, s3-62(1):1--24, 1991.

\bibitem[FM19]{Fink19}
Alex {Fink} and Luca {Moci}.
\newblock {Polyhedra and parameter spaces for matroids over valuation rings.}
\newblock {\em {Adv. Math.}}, 343:448--494, 2019.

\bibitem[{Hus}04]{Husemoeller04}
Dale {Husem\"oller}.
\newblock {\em {Elliptic curves. With appendices by {O}tto {F}orster, {R}uth
  {L}awrence, and {S}tefan {T}heisen. 2nd ed.}}, volume 111.
\newblock New York, NY: Springer, 2nd ed. edition, 2004.

\bibitem[{Ing}71]{Ingleton71}
A.W. {Ingleton}.
\newblock {Representation of matroids.}
\newblock {{C}ombinat. {M}ath. {A}ppl., {P}roc. {C}onf. math. {I}nst., {O}xford
  1969, 149--167}, 1971.

\bibitem[Lan02]{Lang02}
Serge Lang.
\newblock {\em Algebra}.
\newblock Number 211 in Graduate Texts in Mathematics. Springer, rev. 3rd ed
  edition, 2002.

\bibitem[Lin85]{Lindstrom85}
Bernt Lindstr{\"o}m.
\newblock On the algebraic representations of dual matroids.
\newblock Technical Report~5, Department of Math., Univ. of Stockholm, 1985.

\bibitem[Lin86a]{Lindstrom86a}
Bernt Lindstr{\"o}m.
\newblock A non-linear algebraic matroid with infinite characteristic set.
\newblock {\em Discrete Math.}, 59(3):319--320, 1986.

\bibitem[Lin86b]{Lindstrom86b}
Bernt Lindstr{\"o}m.
\newblock The non-{Pappus} matroid is algebraic over any finite field.
\newblock {\em Utilitas Math.}, 30:53--55, 1986.

\bibitem[Lin88]{Lindstrom88}
Bernt Lindstr{\"o}m.
\newblock On {$p$}-polynomial representations of projective geometries in
  algebraic combinatorial geometries.
\newblock {\em Math. Scand.}, 63(1):36--42, 1988.

\bibitem[Pen18]{Pendavingh18}
Rudi Pendavingh.
\newblock Field extensions, derivations, and matroids over skew hyperfields.
\newblock 2018.
\newblock Preprint, \verb+arXiv:1802.02447+.

\bibitem[RST19]{Rosen19}
Zvi Rosen, Jessica Sidman, and Louis Theran.
\newblock Algebraic matroids in action.
\newblock {\em Am. Math. Mon.}, 2019.
\newblock To appear, \verb+arXiv:1809.00865+.

\bibitem[Sch45]{Schilling45}
O.~F.~G. Schilling.
\newblock Noncommutative valuations.
\newblock {\em Bull. Amer. Math. Soc.}, 51(4):297--304, 1945.

\bibitem[{Sil}09]{Silverman09}
Joseph~H. {Silverman}.
\newblock {\em {The arithmetic of elliptic curves. 2nd ed.}}, volume 106.
\newblock New York, NY: Springer, 2nd ed. edition, 2009.

\bibitem[Voi18]{Voight18}
John Voight.
\newblock {\em Quaternion algebras}.
\newblock 2018.
\newblock \verb+https://math.dartmouth.edu/~jvoight/quat.html+, online book,
  v.0.9.14.

\end{thebibliography}

\end{document}